\documentclass[preprint, 12pt]{elsarticle}
\usepackage{amsmath, amsthm, amssymb, mathrsfs, latexsym, bbm, xypic, epsf, setspace}
\usepackage{graphicx}
\newtheorem{thm}{Theorem}[section]
\newtheorem*{claim}{Claim}
\newtheorem*{problem}{Problem}
\newtheorem{defn}[thm]{Definition}
\newtheorem{prop}[thm]{Proposition}
\newtheorem{cor}[thm]{Corollary}
\newtheorem{lem}[thm]{Lemma}
\newtheorem{rem}[thm]{Remark}
\newtheorem{nota}[thm]{Notation}
\def\Ind#1#2{#1\setbox0=\hbox{$#1x$}\kern\wd0\hbox to 0pt{\hss$#1\mid$\hss}
\lower.9\ht0\hbox to 0pt{\hss$#1\smile$\hss}\kern\wd0}
\def\Notind#1#2{#1\setbox0=\hbox{$#1x$}\kern\wd0\hbox to 0pt{\mathchardef
\nn="3236\hss$#1\nn$\kern1.4\wd0\hss}\hbox to 0pt{\hss$#1\mid$\hss}\lower.9\ht0
\hbox to 0pt{\hss$#1\smile$\hss}\kern\wd0}

\def\cl{{\rm cl}}

\def\M{\mathcal{M}}

\def\Cb{\bar{\mathcal{C}}}

\journal{Annals of Pure and Applied Logic}

\begin{document}
\begin{frontmatter}
\title{The geometry of Hrushovski constructions, I. \\ The uncollapsed case.}

\author[UEA]{David M. Evans\corref{cor1}}
\ead{d.evans@uea.ac.uk}
\author[UEA]{Marco S. Ferreira\corref{cor1}}
\ead{masferr@gmail.com}
\address[UEA]{School of Mathematics, UEA, Norwich NR4 7TJ, UK.}
\cortext[cor1]{Corresponding Author}

\begin{abstract}
An intermediate stage in Hrushovski's construction of flat strongly minimal structures in a relational language $L$ produces $\omega$-stable structures of rank $\omega$. We analyze the pregeometries given by forking on the regular type of rank $\omega$ in these structures. We show that varying $L$ can affect the (local) isomorphism type of the pregeometry, but not its finite subpregeometries. A sequel will compare these to the pregeometries of the strongly minimal structures.
\end{abstract}


\end{frontmatter}

\section{Introduction}

Pregeometries arise in model theory via the operation of algebraic closure on a strongly minimal set, or more generally, via forking on the realisations of a regular type (see \cite{AP2}, for example). In the 1980's it was conjectured by Zilber that any strongly minimal structure is geometrically equivalent to one of the `classical' strongly minimal structures in the appropriate language: a pure set, a vector space over a fixed division ring, or an algebraically closed field. Here, `geometrically equivalent' means that the (pre)geometries are isomorphic, possibly after adding a small set of parameters.

This conjecture was refuted by Hrushovski in \cite{EH}. Working with a language $L_3$  having  a single 3-ary relation symbol $R$, Hrushovski constructs continuum-many non-isomorphic (countable, saturated) strongly minimal structures $D_\mu$, where $\mu$ is a function which controls the multiplicities of certain algebraic types. These are not geometrically equivalent to the classical strongly minimal structures: their geometries are not disintegrated as in the pure set case, nor do they embed a group configuration as in the other cases. Hrushovski refers to these structures as `flat' and asks whether there is more than one geometric equivalence type of flat strongly minimal structure; in particular, whether the $D_\mu$ are geometrically equivalent (cf. Section 5.2 of \cite{EH}). Questions of this form are reiterated in \cite{Hasson}. Hrushovski's question is the subject of this paper and its sequel  \cite{MF}.  Before describing our results it will be helpful to remind the reader of some of aspects of Hrushovski's construction, and to fix some notation.

In \cite{EH} the \textit{predimension} of a finite $L_3$-structure is defined to be the size of $A$ minus the number of triples from $A$ which are related by $R$. Of interest is the class $\mathcal{C}_3$ of finite structures where this is non-negative on all substructures. Associated to this is a notion $\leq$ of \textit{self-sufficient} embedding, and a \textit{dimension} $d$ which gives rise to a pregeometry on $A$ (all of this is reviewed in more detail and generality in Section \ref{sec2} here). The class $(\mathcal{C}_3, \leq)$ is an amalgamation class with respect to the distinguished embeddings, and so there is an associated generic structure $\mathcal{M}_3$ (-- sometimes called a \textit{Fra\"{\i}ss\'e limit}) which also inherits a dimension function $d$. It can be shown that $\mathcal{M}_3$ is $\omega$-stable of Morley rank $\omega$ and that the pregeometry $PG(\mathcal{M}_3)$ given by $d$ is exactly the geometry of forking given by the (unique) regular type of rank $\omega$. Similar statements hold if the language is replaced by a language $L_n$ with a single $n$-ary relation symbol, for $n \geq 3$, and we denote the corresponding generic structure by $\mathcal{M}_n$. 

What we have just described is often referred to as the `\textit{uncollapsed} case' of Hrushovski's construction. The strongly minimal $D_\mu$ can be seen as self-sufficient homogeneous substructures of $\mathcal{M}_3$ in which  certain rank-1 types of $\mathcal{M}_3$ have only finitely many realizations. Thus these rank-1 types are `collapsed' to algebraic types in $D_\mu$. The strongly minimal structure $D_\mu$ is also constructed as the generic structure of an amalgamation class $(\mathcal{C}_\mu, \leq)$ where $\mathcal{C}_\mu$ is a subclass of $\mathcal{C}_3$ in which the multiplicities of certain quantifier-free types (minimally simply algebraic extensions) are bounded by the function $\mu$.

We can now describe the structure and main results of this paper. Section 2 contains background material on the (uncollapsed) construction. In Section 3 we find embeddings between the pregeometries of the $\mathcal{M}_n$ and show that these pregeometries  cannot be distinguished by finite subgeometries (Corollary \ref{cor39}). However, we show in Section 4 that if $m \neq n$, then the (pre)geometries of $\mathcal{M}_n$ and $\mathcal{M}_m$ are \textit{not} isomorphic (Theorem \ref{T527}).  In Section 5 we examine the effect of localization on the pregeometies $PG(\mathcal{M}_n)$ and show (Theorem \ref{T532}) that the geometry after localizing over a finite subset is isomorphic to the original geometry. Combining this with the results of Section 3, we deduce that if $m \neq n$, then the (pre)geometries of $\mathcal{M}_n$ and $\mathcal{M}_m$ are not \textit{locally} isomorphic. The main result of Section 6 is to show that the pregeometry $PG(\mathcal{M}_n)$ can be viewed as the generic structure of an amalgamation class $(P_n, \unlhd_n)$ of pregeometries (Theorem \ref{T543}). 

The main ingredient in the proof of these results is a series of `Changing Lemmas.' Typically, these describe the effect (on the pregeometry) of replacing part of a structure in some $\mathcal{C}_n$ by another structure in $\mathcal{C}_n$ with the same domain. Related results are used  in \cite{DE1} and \cite{DE2}.

In the sequel \cite{MF} (and in \cite{MFThesis}) we show that the pregeometries of the strongly minimal sets $D_\mu$ are all isomorphic to the pregeometry of $\mathcal{M}_3$. Moreover, the variation on the strongly minimal set construction given in 5.2 of \cite{EH} produces pregeometries locally isomorphic to the pregeometry of $\mathcal{M}_3$. Parallel results hold if the construction is done with an $n$-ary relation in place of a $3$-ary relation (and more general languages), and it appears that the distinct geometric equivalence types of the flat countable saturated strongly minimal structures produced by the construction in \cite{EH} are given by the pregeometries $PG(\mathcal{M}_n)$ (for $n\geq 3$).

\medskip

\noindent\textit{Acknowledgement:\/} Most of the results of this paper were produced whilst the second Author was supported as an Early Stage Researcher by the Marie Curie Research Training Network MODNET, funded by grant MRTN-CT-2004-512234 MODNET from the CEC.

\section{The construction and notation}\label{sec2}

We start with a brief description of the objects of our study, following Hrushovski \cite{EH} and Wagner \cite{FW}. The article \cite{B&S} of Baldwin and Shi gives another axiomatization and generalization of this method. The book \cite{AP2} of Pillay contains all necessary background material on pregeometries and model theory.

\smallskip

We work in slightly more generality than that outlined in the Introduction. 

\begin{defn}\rm\label{fnotation}
Let $I$ be a countable set and $f:I\to\mathbbm N\setminus\{0\}\times\mathbbm N\setminus\{0\}$ be a function. Write  $f(i)=(n_i,\alpha_i)$. Let $L_f=\{R_i:i\in I\}$ be a language where $R_i$ is an $n_i$-ary relational symbol. 

If $A$ is a finite $L_f$-structure, we define its \textit{predimension}  $\delta_f(A)$ to be
$$\delta_f(A)=|A|-\sum_{i\in I}\alpha_i|R_i^A|$$
(or $-\infty$), where $R_i^A$ is the set of $n_i$-tuples from $A$ which satisfy the relation $R_i$. We let 
$\mathcal C_f$ be the set of finite $L_f$-structures $A$ such that $\delta_f(A') \geq 0$ for all $A'  \subseteq A$. 

Suppose $A \subseteq B \in \mathcal{C}_f$.  We write $A \leq B$ and say that $A$ is \textit{self-sufficient} in $B$ (or \textit{strong} in $B$) if for all $B'$ with $A \subseteq B' \subseteq B$ we have $\delta_f(A) \leq \delta_f(B')$. In particular, $\emptyset \leq B$ for $B \in \mathcal C_f$.
\end{defn}

A key property of $\delta_f$ is that it is \textit{submodular}: if $A, B \subseteq C \in \mathcal{C}_f$,  then $\delta_f(A \cup B) \leq \delta_f(A) + \delta_f(B) - \delta_f(A \cap B)$. Using this, one shows that if $A \leq B \leq C$ then $A \leq C$, and if $A_1, A_2 \leq B$ then $A_1\cap A_2 \leq B$. 

Let $\bar {\mathcal C_f}$  be the class of $L_f$-structures all of whose finite substructures lie in $\mathcal{C}_f$. We can extend the notion of self-sufficiency to this class in a natural way, and for $B \subseteq C \in \bar {\mathcal C_f}$ we have $B \leq C$ if and only if whenever $A \leq B$ is finite, then $A \leq C$. 

Note that if $A \subseteq B \in \bar {\mathcal C_f}$  and $A$ is finite then there is a finite $A'$ with $A \subseteq A' \subseteq B$ and $\delta_f(A')$ as small as possible. In this case $A' \leq B$. It follows that there is a smallest finite set $C \leq B$ with $A \subseteq C$. We denote this by $\cl_B(A)$ and refer to it as the \textit{self-sufficient closure} of $A$ in $B$. Thus the self-sufficient closure (in some given structure) of a finite set is finite, and we can extend the definition to arbitrary subsets of a structure in $\bar{\mathcal{C}}_f$. It is easy to see that $\cl_B$ is indeed a closure operation on $B$. However it does not usually satisfy the exchange condition and it is \textit{not} the closure operation which produces our pregeometries. The latter is what is given in the following. Note that we write $\delta$ rather than $\delta_f$ where no confusion will arise.

\begin{defn} \rm
Let $\mathcal M\in\bar {\mathcal C_f}$ and $A$ be a finite subset of $\mathcal{M}$. We define the \textit{dimension} $d_{\mathcal M}(A)$ of $A$ (in $\mathcal{M}$)  to be the minimum value of $\delta(A')$ for all finite subsets $A'$ of $\mathcal{M}$ which contain $A$. 

We define the $d$-\textit{closure} of $A$ in $\mathcal{M}$ to be: 
$$\cl^{d_{\mathcal M}}(A)=\{c\in\mathcal M:d_{\mathcal M}(Ac)=d_{\mathcal M}(A)\}.$$ 

We can coherently extend the definition of $d$-closure to infinite subsets $A$ of $\mathcal M$ by saying that the $d$-closure of $A$ is the union of the $d$-closures of finite subsets of $A$. 
\end{defn}

Note  that the $d$-closure of a finite set does not need to be finite. Concerning the relation with the self-sufficient closure, we always have $\cl_{\mathcal M}(A)\subseteq \cl^{d_{\mathcal M}}(A)$. If $A$ is finite then $d_{\mathcal M}(A)=\delta(\cl_{\mathcal M}(A))$. Also if $A$ is finite we have $A\leq\mathcal M\text{ if and only if }d_{\mathcal M}(A)=\delta(A)$. 

\begin{thm}
Let $\mathcal M\in\bar {\mathcal C_f}$. Then $(\mathcal M,\cl^{d_{\mathcal M}})$ is a pregeometry. Moreover, the dimension function (as cardinality of a basis) equals $d_{\mathcal M}$ on finite subsets of $\mathcal M$. We may use the notation $PG(\mathcal M)$ instead of $(\mathcal M,\cl^{d_{\mathcal M}})$.
\end{thm}

The class $(\mathcal C_f,\leq)$ has the following (strong) $\leq$-\textit{free amalgamation property}.

\begin{prop}  Let $A_0\subseteq A_1\in\mathcal C_f$ and $A_0\leq A_2\in\mathcal C_f$ and $A_1\cap A_2=A_0$. Let $F = A_1\amalg_{A_0}A_2$ be the structure with underlying set $A_1\cup A_2$ such that the only relations are the ones arising from $A_1$ and $A_2$. Then $F\in\mathcal C_f$ and $A_1\leq F.$
\end{prop}

\begin{defn}\rm Let $\mathcal M$ be an $L_f$-structure. We say that $\mathcal M$ is a \textit{generic structure for} $(\mathcal C_f,\leq)$ if it is countable and satisfies the following conditions:
\begin{enumerate}
\item [(F1)] $\mathcal M\in\bar{\mathcal C_f}$
\item [(F2)] (extension property) If $A\leq\mathcal M$ and $A\leq B\in\mathcal C_f$ then there exists an embedding $g:B\to\mathcal M$ such that $g_{|A}=Id_{|A}$ (where $Id_{|A}$ is the identity map) and such that $g(B)\leq\mathcal M$.
\end{enumerate}
\end{defn}

Here we could also replace (F1) by
\begin{enumerate}
\item[(F1$'$)] $\mathcal{M}$ is the union of a chain $A_0 \leq A_1 \leq A_2 \leq \cdots$ of structures in $\mathcal{C}_f$.
\end{enumerate}

A standard argument shows:

\begin{thm}
There is a generic model $\mathcal M_{\mathcal C_f}$ for $(\mathcal C_f,\leq)$. It is unique up to isomorphism and if $h : A_1 \to A_2$ is an isomorphism beteween finite self-sufficient substructures, then $h$ extends to an automorphism of  $\mathcal M_{\mathcal C_f}$.  We may use the notation $\mathcal M_f$ instead of $\mathcal M_{\mathcal C_f}$.
\end{thm}

\begin{prop} 
\label{P514}
With the above notation, If $L_f$ is finite then $\mathcal M_f$ is saturated and $Th(\mathcal M_f)$ is $\omega$-stable.
\end{prop}

\begin{rem}\rm
There are examples of $f$ where $L_f$ is not finite and such that $\mathcal M_f$ is saturated and $Th(\mathcal M_f)$ is $\omega$-stable, for example when $\mathcal C_f$ has only finitely many isomorphism types of each size. However we believe that we do get saturation and $\omega$-stability for every choice of $f$ because $f$ is integer valued. But this would require further analysis.
\end{rem}

Among these examples we would like to distinguish the following ones. The first case is as in the Introduction:
\begin{nota}\rm
\begin{enumerate}
\item[(i)] Let $n \geq 3$ and $I$ consist of a singleton. The language $L_n$ consists of a single $n$-ary relation symbol $R$ and the predimension is given by $\delta_n(A) = \vert A \vert - \vert R^A\vert$. In this case we denote the class by $\mathcal{C}_n$ and the generic model by $\mathcal{M}_n$. 
\item[(ii)] Let $I=\mathbbm N\setminus\{0\}$ and $f$ a function defined by $f(i)=(i,1)$ for each $i\in I$. In other words we are considering the predimension given by: $$\delta_f(A)=|A|-\sum_{i\in\mathbbm N\setminus\{0\}}|R_i^A|,$$
where $R_i$ is an $i$-ary relation.
We  write $L_{\omega}$, $\mathcal C_{\omega}$, $\delta_{\omega}$, $d_{\omega}$ and $\mathcal M_{\omega}$ instead of $L_f$, $\mathcal C_f$, $\delta_f$, $d_f$ and $\mathcal M_f$, for this particular function $f$.
\end{enumerate}
\end{nota}

It will be convenient to fix a first order language for the class of pregeometries.  A reasonable choice for this is the language $LPI=\{I_n:n\in\mathbbm N\setminus\{0\}\}$ where each $I_n$ is an $n$-ary relational symbol. A pregeometry $(P,\cl^P)$ will be seen as a structure in this language by saying $I_n^P=\{\bar a\in P^n:\bar a\textit{ is independent in }P\}$. Notice that we can recover a pregeometry just by knowing its finite independent sets. In particular,  a pregeometry is completely determined by the dimension function on its  finite subsets. Note that the isomorphism type of a pregeometry is determined by the isomorphism type of its associated geometry and the size of the equivalence classes of interdependence. In the case where these are all countably infinite, it therefore makes no difference whether we consider the geometry or the pregeometry.

\section{Pregeometries from different predimensions}

The first result of this section shows that sometimes the pregeometry on the generic model associated to a particular predimension is isomorphic to the pregeometry on the generic model associated to a `simpler' predimension. This is motivated by the observation that the pregeometry associated to the predimension $\delta(A)=|A|-|R_3^A|-|R_4^A|$ is isomorphic to the simpler one given by $\delta(A)=|A|-|R_4^A|$. Actually, the proof of the following general result follows the same idea as the proof of this particular result, but is considerably more technical. 

\begin{thm}
\label{T516}
Let $I$ be a countable set. Let $f:I\to\mathbbm N\setminus\{0\}\times\mathbbm N\setminus\{0\}$ be a function and $f(i)=(n_i,\alpha_i)$. Let $\sim$ be an equivalence relation on $I$ defined by $i\sim j$ if and only if $f(i)=f(j)$. Define a partially ordered set $(\widetilde I,\leq)$ by saying $[i]_\sim\leq[j]_\sim$ if and only if $n_i\leq n_j$ and $\alpha_j|\alpha_i$. Let $J\subseteq I$ be such that $\widetilde J$ is cofinal in $\widetilde I$ (so this means that for every $i \in I$ there exists $j \in J$ with $n_i \leq n_j$ and $\alpha_j \vert \alpha_i$). Then $$PG(\mathcal M_f)\simeq PG(\mathcal M_{f_{|J}}).$$
\end{thm}

\begin{proof} We start by outlining the idea of the proof. We construct a countable structure $\mathcal M_f^{\pi}$ in the restricted language $L_{f_{|J}}$ with the same underlying set as $\mathcal M_f$ in such a way that $PG(\mathcal M_f^{\pi})=PG(\mathcal M_f)$ and such that $\mathcal M_f^{\pi}$ is isomorphic to $\mathcal M_{f_{|J}}$. We construct $\mathcal M_f^{\pi}$ in such a way that for any finite subset $A$ the predimension value remains the same as the predimension value of $A$ calculated in $\mathcal M_f$, in particular the pregeometry remains unchanged.

To achieve this we replace each tuple in $R_i^{\mathcal M_f}$, where $i\in I\setminus J$, by some new $R_j$-tuples, where $j\in J$, and the number of replacement tuples for such a given tuple is exactly $\alpha_i/\alpha_j$, to compensate for the diferent weights. We do this replacement process in steps: first we need to decide which $j\in J$ to use for a given $i\in I\setminus J$, then in each step we fix $j$ and a finite subset $A$ of $\mathcal M_f$ then for each such $i$ corresponding to $j$, we replace each $R_i$-tuple with underlying set $A$ by $\alpha_i/\alpha_j$ new $R_j$-tuples, each one of them with underlying set $A$. An adequate $j$ will be such that $n_j\geq n_i$ and $\alpha_j\vert\alpha_i$. Moreover we need to check that we have enough new replacement tuples to choose from: basically permutation of coordinates will produce enough such tuples, together with the fact that $\emptyset\leq\mathcal M_f$, which assures that we do not have too many `old' tuples to start with. The details are as follows.

Let $h:I\to J$ be a function such that $[i]_\sim\leq[h(i)]_\sim$ and such that $j\in J\Rightarrow h(j)=j$. Such a function exists because $\widetilde J$ is cofinal in $\widetilde I$.
Let $g:\bigcup_{n=1}^{\infty}\mathcal M_f^n\to\mathbbm N\setminus\{0\}$ be defined by $g(a_1,\cdots,a_n)=|\{a_1,\cdots,a_n\}|$. 
Observe that for each $A\subseteq\mathcal M_f$ and $n,m\in\mathbbm N\setminus \{0\}$, if $A^n\cap g^{-1}(m)$ is nonempty then, $$|A^n\cap g^{-1}(m)|\geq m!\geq m$$ because we can permute the $m$ distinct coordinates of an $n$-tuple in $A^n\cap g^{-1}(m)$ obtaining $m!$ tuples.

Now we fix $j\in J$. Let $A\subseteq\mathcal M_f$ be such that $|A|\leq n_j$. Notice that $R_i^A\cap g^{-1}(|A|)$ is exactly the set of $R_i^{\mathcal M_f}$ tuples with underlying set $A$. Our aim is to replace each tuple $\bar a\in R_i^A\cap g^{-1}(|A|)$ for all $i\in h^{-1}(j)\setminus\{j\}$ by one or more tuples in $(A^{n_j}\setminus R_j^A)\cap g^{-1}(|A|)$. We need to prove that we have enough space for this. If $\delta_f(A)\geq 0$ then $\sum_{i\in I}\alpha_i|R_i^A|\leq|A|$, hence $\sum_{i\in h^{-1}(j)}\alpha_i|R_i^A|\leq|A|$ and so $\sum_{i\in h^{-1}(j)\setminus\{j\}}\alpha_i|R_i^A|\leq|A|-\alpha_j|R_j^A|$. But then we have
\begin{enumerate}
\item []$\sum_{i\in h^{-1}(j)\setminus\{j\}}\alpha_i|R_i^A\cap g^{-1}(|A|)|\leq|A|-\alpha_j|R_j^A|$
\item []$\leq|A^{n_j}\cap g^{-1}(|A|)|-\alpha_j|R_j^A|$
\item []$\leq\alpha_j(|A^{n_j}\cap g^{-1}(|A|)|-|R_j^A\cap g^{-1}(|A|)|)$
\item []$=\alpha_j|(A^{n_j}\setminus R_j^A)\cap g^{-1}(|A|)|.$
\end{enumerate}
Thus we have $$\sum_{i\in h^{-1}(j)\setminus\{j\}}\frac{\alpha_i}{\alpha_j}|R_i^A\cap g^{-1}(|A|)|\leq|(A^{n_j}\setminus R_j^A)\cap g^{-1}(|A|)|.$$

Observe that the first term in the above inequality is the number of replacement tuples that are needed for this step and the second term is the number of available new tuples to use as replacements. So we do have space, that is, for each $i\in h^{-1}(j)\setminus\{j\}$ and $\bar a\in R_i^A\cap g^{-1}(|A|)$ we can replace $\bar a$ by $\frac{\alpha_i}{\alpha_j}$ distinct tuples in $(A^{n_j}\setminus R_j^A)\cap g^{-1}(|A|)$.

So now, if $i \in I\setminus J$ and $j = h(i)$ we can replace an $R_i$-tuple $\bar a \in R_i^{\mathcal{M}_f}$ by $\alpha_i/\alpha_j$ $R_j$-tuples $\bar{a}' \in \mathcal{M}_f^{n_j} \setminus R_j^{\mathcal{M}_f}$ with the same underlying set as $\bar{a}$. Doing this does not change the predimension of any subset of $\mathcal{M}_f$. The above calculation shows that we can do this (in an inductive fashion) for all steps corresponding to a choice of $j\in J$ and underlying set $A$, replacing all $\bar{a} \in R_i^{\mathcal{M}_f}$ for all $i\in I\setminus J$. We obtain a structure $\mathcal M_f^\pi$ in the restricted language $L_{f_{|J}}$ with the same underlying set as $\mathcal M_f$ and the same predimension function on finite subsets.

Now let $\pi:\mathcal M_f\to\mathcal M_f^{\pi}$ be the identity function. If $A\subseteq \mathcal M_f$ then $\pi(A)$ means the substructure of $\mathcal M_f^{\pi}$ with $A$ as the underlying set. Reciprocally if $A\subseteq \mathcal M_f^{\pi}$ then $\pi^{-1}(A)$ means the substructure of $\mathcal M_f$ with $A$ as the underlying set. The construction of $\mathcal M_f^{\pi}$ was made so that for each finite $A\subseteq\mathcal M_f$ we have: $$\delta_f(A)=\delta_{f_{|J}}(\pi(A)).$$ 
In other words, for each finite $A\subseteq\mathcal M_f^{\pi}$ we have: $$\delta_{f_{|J}}(A)=\delta_f(\pi^{-1}(A)).$$

In particular, for each finite $A\subseteq\mathcal M_f^{\pi}$ we have $\delta_{f_{|J}}(A)=\delta_f(\pi^{-1}(A))\geq 0$, thus $\mathcal M_f^{\pi}\in\bar {\mathcal C}_{f_{|J}}$. Also, because the predimension function is the same, so is the dimension function. Thus $$PG(\mathcal M_f^{\pi})\simeq PG(\mathcal M_f).$$

We now want to prove that $\mathcal M_f^{\pi}$ is the generic model for the class $\mathcal C_{f_{|J}}$, that is, isomorphic to $\mathcal M_{f_{|J}}$. It remains to show the extension property (the axiom $F2$). The diagram below will help us to follow the proof.

\[
\begin{xy}
\xymatrix{
& & & \mathcal M_f\ar[r]^{\pi} & \mathcal M_f^{\pi}\\
B\ar[r]^{l_1} & B'\ar[r]^{l_2} & l_2(B')\ar[r]^{\pi}\ar[ru]^{\leq} & \pi l_2(B')\ar[ru]^{\leq}\\
A\ar[r]\ar[u]^{\leq} & \pi^{-1}(A)\ar[r]\ar[u]^{\leq}\ar[ru]^{\leq} & A\ar[ru]^{\leq}\\
}
\end{xy}
\]

Let $A\leq B\in\mathcal C_{f_{|J}}\subseteq \mathcal C_f$ and $A\leq\mathcal M_f^{\pi}$.  Let $B'$ be obtained from $B$ just by replacing $A$ by $\pi^{-1}(A)$ and nothing else. Let $l_1:B\to B'$ be the identity function. We still have $\pi^{-1}(A)\leq B'\in\mathcal C_f$ and $\pi^{-1}(A)\leq\mathcal M_f$. In fact we have that $A\leq B$ implies that $\pi^{-1}(A)\leq B'$ because the relations that we add when going from $\pi^{-1}(A)$ to $B'$ are the same we add when going from $A$ to $B$. In particular we get $\emptyset\leq\pi^{-1}(A)\leq B'$ so we may conclude that $B'\in\mathcal C_f$. Also $A\leq \mathcal M_f^{\pi}$ implies $\pi^{-1}(A)\leq\mathcal M_f$ because when going from $\mathcal M_f$ to $\mathcal M_f^{\pi}$ we do not change the predimension values of subsets. 

So we have $\pi^{-1}(A)\leq B'\in\mathcal C_f$ and $\pi^{-1}(A)\leq\mathcal M_f$ and we can apply the extension property of $\mathcal M_f$. Let $l_2:B'\to\mathcal M_f$ be an embedding such that $l_2(B')\leq\mathcal M_f$ and $l_{2|\pi^{-1}(A)}=Id_{|\pi^{-1}(A)}$. Now we apply the $\pi$ function to the chain $\pi^{-1}(A)\leq l_2(B')\leq\mathcal M_f$ obtaining $A\leq\pi l_2 l_1(B)\leq\mathcal M_f^{\pi}$ and we are done.

It is not completely obvious that $\pi l_2l_1$ is an $L_{f_{|J}}$-embedding, so let us check that.

Let $j\in J$. We want to prove that $\bar a\in R_j^B$ if and only if $\pi l_2l_1(\bar a)\in R_j^{\mathcal M_f^{\pi}}$. For $\bar a\subseteq A$ this is clear. For $\bar a\nsubseteq A$ then we observe that:
\begin{enumerate}
\item [] $\bar a\in R_j^B\setminus R_j^A\Leftrightarrow l_1(\bar a)\in R_j^{B'}\setminus R_j^{\pi^{-1}(A)}$ (the changes are made inside $A$)
\item [] $\Leftrightarrow l_2l_1(\bar a)\in R_j^{\mathcal M_f}\setminus R_j^{\pi^{-1}(A)}$ ($l_2$ is an embedding)
\item [] $\Rightarrow\pi l_2l_1(\bar a)\in R_j^{\mathcal M_f^{\pi}}\setminus R_j^A$ (because $j\in J$).
\end{enumerate}
 
But the converse of the last implication also holds. This is because for $i\in I\setminus J$ we have by construction $R_i^B=\emptyset\Rightarrow R_i^{B'}\setminus R_i^{\pi^{-1}(A)}=\emptyset\Leftrightarrow R_i^{l_2(B')}\setminus R_i^{\pi^{-1}(A)}=\emptyset$ so when we apply $\pi$ no tuples are added from $R_j^{l_2(B')}\setminus R_j^{\pi^{-1}(A)}$ to $R_j^{\pi l_2(B')}\setminus R_j^A$. Thus $\pi l_2l_1$ is an embedding.

Finally, by the uniqueness of the generic model we have $\mathcal M_f^{\pi}\simeq\mathcal M_{f_{|J}}$. In particular $PG(\mathcal M_f^{\pi})\simeq PG(\mathcal M_{f_{|J}})$. But we already know that $PG(\mathcal M_f^{\pi})=PG(\mathcal M_f)$, thus we have $PG(\mathcal M_f)\simeq PG(\mathcal M_{f_{|J}})$.
\end{proof}

Note that the proof of this result depends on the fact that we are working with ordered tuples, however a  similar result working with sets rather than tuples can be obtained (cf. Section 4.3 in \cite{MF}).  Of course, we can also ask about the extent to which the hypotheses are necessary. We obtain the following consequences the previous result.

\begin{cor} 
\label{C517}
Let $I$ be a countable set and $f:I\to\mathbbm N\setminus\{0\}\times\mathbbm N\setminus\{0\}$ be a function. Then $PG(\mathcal M_f)$ embeds in $PG(\mathcal M_{\omega})$. In other words, the pregeometry associated to the predimension $$\delta_f(A)=|A|-\sum_{i\in I}\alpha_i|R_i^A|$$ embeds into the pregeometry associated to the predimension $$\delta_{\omega}(A)=|A|-\sum_{n\geq 1}|R_n^A|.$$
\end{cor}
\begin{proof}
We extend $I$ to $\bar I=I\cup I_{\omega}$ with $I_{\omega}=\mathbbm N\setminus\{0\}$ (we can assume $I\cap I_{\omega}=\emptyset$). Then we extend $f$ to $\bar f:\bar I\to\mathbbm N\setminus\{0\}\times\mathbbm N\setminus\{0\}$ by saying $\bar f(i)=(i,1)$ for $i\in I_{\omega}$. Clearly we have that $\mathcal M_f$ embeds (as a closed substructure) in $\mathcal M_{\bar f}$. To be more precise $\mathcal M_f$ can be seen as a structure in $\bar{ \mathcal C}_{\bar f}$ if we interpret the extra relational symbols as the empty set, then we use the extension property of $\mathcal M_{\bar f}$ recursively because $\mathcal M_f$ is countable. Thus $PG(\mathcal M_f)$ embeds in $PG(\mathcal M_{\bar f})$. Now we observe that by Theorem \ref{T516} we have $PG(\mathcal M_{\bar f})\simeq PG(\mathcal M_{\bar f_{|I_{\omega}}})$. But $\mathcal M_{\bar f_{|I_{\omega}}}\simeq\mathcal M_{\omega}$, thus $PG(\mathcal M_{\bar f})\simeq PG(\mathcal M_{\omega})$ and $PG(\mathcal M_f)$ embeds in $PG(\mathcal M_{\omega})$.
\end{proof}

\begin{cor}
Consider the languages $L_{3,4}=\{R_3,R_4\}$ and $L_4=\{R_4\}$. Associated to the language $L_4$ we have the standard construction using the predimension $\delta_4(A)=|A|-|R_4^A|$ obtaining the class $\mathcal C_4$ and the generic $\mathcal M_4$. Associated to the language $L_{3,4}$ we proceed in the same manner as in the standard case using the predimension $\delta_{3,4}(A)=|A|-|R_3^A|-|R_4^A|$, obtaining the class $\mathcal C_{3,4}$ and the generic $\mathcal M_{3,4}$. Then, $$PG(\mathcal M_{3,4})\simeq PG(\mathcal M_4).$$
\end{cor}
\begin{proof}
In the notation of Theorem \ref{T516} we just observe that $J=\{4\}$ is cofinal in $I=\{3,4\}$.
\end{proof}

Now we want to prove further that $PG(\mathcal M_{\omega})$ embeds in $PG(\mathcal M_3)$. First we need the following lemma.

\begin{nota} \rm
Let $\mathcal M,\mathcal N$ be two $L$-structures and $A$ be a subset of both $\mathcal M$ and $\mathcal N$. We write $A[\mathcal M]$ to denote the substructure of $\mathcal M$ with underlying set $A$ and $A[\mathcal N]$ similarly.
\end{nota}

\begin{lem}
\label{L520}
Let $\mathcal M_{\omega(3)}$ be the generic model associated to the predimension $$\delta_{\omega(3)}(A)=|A|-\sum_{n\geq 3}|R_n^A|.$$ Then, $$PG(\mathcal M_{\omega(3)})\textit{ embeds in }PG(\mathcal M_3).$$
\end{lem}

\begin{proof}
Let $\mathcal M$ be an $L_{\omega(3)}$-structure and $\bar b\in T_4(\mathcal M):=\bigcup_{n\geq 4}R_n^{\mathcal M}$. We define an $L_{\omega(3)}$-structure $\mathcal M^{\bar b}$ obtained from $\mathcal M$ by replacing the tuple $\bar b=(b_1,\cdots,b_n)$ by the relations $$\{(b_1,x_1,b_2),(b_2,x_2,b_3),\cdots,(b_{n-1},x_{n-1},b_n)\}\cup\{(x_1,\cdots,x_{n-1})\}$$ where $\{x_1,\cdots,x_{n-1}\}$ are new distinct points specifically added for this task. So as a set $\mathcal M^{\bar b}=\mathcal M\cup\{x_1,\cdots,x_{n-1}\}$. We say that $\bar b'=(x_1,\cdots,x_{n-1})$ is the derivative of $\bar b$.
\begin{claim}
Let $\mathcal M$ be an $L_{\omega(3)}$-structure and $\bar b=(b_1,\cdots,b_n)\in T_4(\mathcal M)$. Let $A'$ be a finite subset of $\mathcal M^{\bar b}$ and let $\bar b'=(x_1,\cdots,x_{n-1})$ and $A'=A\cup X$ with $A=A'\cap\mathcal M$ and $X=A'\cap\{x_1,\cdots,x_{n-1}\}$. Then, $$\delta_{\omega(3)}(A')\geq\delta_{\omega(3)}(A[\mathcal M]).$$
In particular, if $\mathcal M\in\bar {\mathcal C}_{\omega(3)}$ then $\mathcal M^{\bar b}\in\bar {\mathcal C}_{\omega(3)}$.
\end{claim}
\begin{proof}
Suppose $X\neq\{x_1,\cdots,x_{n-1}\}$. Then, when going from $A[\mathcal M]$ to $A'$ we added $|X|$ points and added at most $|X|$ relations. This is because the relation $(x_1,\cdots,x_{n-1})$ is not added as $X\neq\{x_1,\cdots,x_{n-1}\}$. Thus, in this case $\delta_{\omega(3)}(A')\geq\delta_{\omega(3)}(A[\mathcal M])$.

Suppose that $X=\{x_1,\cdots,x_{n-1}\}$ and $\bar b\nsubseteq A$. Then, when going from $A[\mathcal M]$ to $A'$ we added $n-1$ points and at most $n-1$ relations because as $\bar b\nsubseteq A$, one of the relations $(b_k,x_k,b_{k+1})$ is not added. Thus in this case we have also $\delta_{\omega(3)}(A')\geq\delta_{\omega(3)}(A[\mathcal M])$.

Finally, suppose that $X=\{x_1,\cdots,x_{n-1}\}$ and $\bar b\subseteq A$. Then, when going from $A[\mathcal M]$ to $A'$ we added $n-1$ points and exactly $n$ relations, but we removed the relation $\bar b=(b_1,\cdots,b_n)$. Thus in this case $\delta_{\omega(3)}(A')=\delta_{\omega(3)}(A[\mathcal M])$.

Thus we conclude that $\delta_{\omega(3)}(A')\geq\delta_{\omega(3)}(A[\mathcal M])$ for all finite $A'\subseteq\mathcal M^{\bar b}$. In particular, if $\mathcal M\in\bar {\mathcal C}_{\omega(3)}$ then $\mathcal M^{\bar b}\in\bar {\mathcal C}_{\omega(3)}$. 
\end{proof}
Now we need to prove another claim.

\begin{claim}
Let $\mathcal M\in\bar{\mathcal C}_{\omega(3)}$ and $\bar b\in T_4(\mathcal M)$. Then the inclusion $\mathcal M\to\mathcal M^{\bar b}$ gives an embedding of pregeometries (for example in the $LPI$ language).
\end{claim}
\begin{proof}
We are going to prove that for finite $A\subseteq\mathcal M$ we have $d_{\mathcal M}(A)=d_{\mathcal M^{\bar b}}(A)$.

Let $A\subseteq B'\subseteq\mathcal M^{\bar b}$ with $B'$ finite and $B'=B\cup X$ as in the previous claim. Then $A\subseteq B$ and by the previous claim we have $\delta_{\omega(3)}(B[\mathcal M])\leq\delta_{\omega(3)}(B')$. This shows that $d_{\mathcal M}(A)\leq d_{\mathcal M^{\bar b}}(A)$.

Conversely, let $A\subseteq B\subseteq\mathcal M$ with $B$ finite. If $\bar b\nsubseteq B$ then $\delta_{\omega(3)}(B[\mathcal M])=\delta_{\omega(3)}(B[\mathcal M^{\bar b}])$. If $\bar b\subseteq B$ then we put $B'=B[\mathcal M^{\bar b}]\cup\{x_1,\cdots,x_{n-1}\}$ where $\bar b'=(x_1,\cdots,x_{n-1})$. Then $\delta_{\omega(3)}(B[\mathcal M])=\delta_{\omega(3)}(B')$ as when going from $B[\mathcal M]$ to $B'$ we added $n-1$ points, $n$ relations and remove the relation $\bar b$. 

The fact that we can find a substructure of $\mathcal M^{\bar b}$ containing $A$ with the $\delta_{\omega(3)}$ value equal to $\delta_{\omega(3)}(B[\mathcal M])$ shows that $d_{\mathcal M^{\bar b}}(A)\leq d_{\mathcal M}(A)$. Thus we have $d_{\mathcal M}(A)=d_{\mathcal M^{\bar b}}(A)$ and this proves that the inclusion $\mathcal M\to\mathcal M^{\bar b}$ is an embedding of pregeometries.
\end{proof}

Now we make a construction by recursion, of a sequence of $L_{\omega(3)}$-structures. We start by ordering the set of tuples $T:=T_4(\mathcal M_{\omega(3)})=\bigcup_{n\geq 4}R_n^{\mathcal M_{\omega(3)}}$.

Let $\mathcal M^{(0)}=\mathcal M_{\omega(3)}$. Let $\mathcal M^{(1)}=(\mathcal M^{(0)})^{\bar c_0}$ where $\bar c_0$ is the first element of $T$.

Suppose that we have constructed $\mathcal M^{(i+1)}=(\mathcal M^{(i)})^{\bar c_i}$. Now if $\bar c_i$ has $n$-arity greater or equal than $4$, then we set $\bar c_{i+1}=\bar c_i'$, otherwise we set $\bar c_{i+1}$ equal to the next unused element of $T$. Finally we put $\mathcal M^{(i+2)}=(\mathcal M^{(i+1)})^{\bar c_{i+1}}$.

We have constructed a sequence of $L_{\omega(3)}$-structures $\mathcal M^{(i)}$ with $i\in\omega$. Moreover, if we iterate this procedure $\omega$ steps we end up with a structure $\mathcal M^{(\omega)}$ only with $3$-tuples. By the previous claims $\mathcal M^{(i)}\in\bar{\mathcal C}_{\omega(3)}$ and for each $i\in\omega$ the inclusion function $\mathcal M^{(i)}\to\mathcal M^{(i+1)}$ gives an embedding of pregeometries. Finally we can define a pregeometry $P^{(\omega)}=\bigcup_{i\in\omega}PG(\mathcal M^{(i)})$. Let $d^{(\omega)}$ be the dimension function of $P^{(\omega)}$.

We need to prove the following claim.
\begin{claim}
Given a finite $A\subseteq\mathcal M^{(\omega)}$ then for all $i$ large enough we have $A[\mathcal M^{(i)}]=A[\mathcal M^{(\omega)}]$. In particular $\mathcal M^{(\omega)}\in\bar{\mathcal C}_3\subseteq\bar{\mathcal C}_{\omega(3)}$.
\end{claim}
\begin{proof}
Let $i_0$ be large enough so that $A\subseteq\mathcal M^{(i_0)}$. Now $A[\mathcal M^{(i_0)}]\in\mathcal C_{\omega(3)}$ so $|A|-\sum_{n\geq 4}|R_n^{A[\mathcal M^{(i_0)}]}|\geq 0$, in particular $\bigcup_{n\geq 4}R_n^{A[\mathcal M^{(i_0)}]}$ is finite. Thus, by our construction there is $i_1\geq i_0$ such that $A[\mathcal M^{(i_1)}]$ has only $3$-tuples. Now it is clear that for all $i\geq i_1$ we have $A[\mathcal M^{(i)}]=A[\mathcal M^{(\omega)}]$. In particular $\mathcal M^{(\omega)}\in\bar{\mathcal C}_3\subseteq\bar{\mathcal C}_{\omega(3)}$. 
\end{proof}

We need to prove one last claim.
\begin{claim}
$P^{(\omega)}=PG(\mathcal M^{(\omega)})$.
\end{claim}
\begin{proof}
Let $A\subseteq\mathcal M^{(\omega)}$ and $A$ finite. We want to prove that $d^{(\omega)}(A)=d_{\mathcal M^{(\omega)}}(A)$. By our last claim there is $i$ such that $\cl_{\mathcal M^{(\omega)}}(A)[\mathcal M^{(\omega)}]=\cl_{\mathcal M^{(\omega)}}(A)[\mathcal M^{(i)}]$. Then,
\begin{enumerate}
\item[] $d^{(\omega)}(A)=d_{\mathcal M^{(i)}}(A)$
\item[] $\leq d_{\mathcal M^{(i)}}(\cl_{\mathcal M^{(\omega)}}(A))$
\item[] $\leq\delta_{\omega(3)}(\cl_{\mathcal M^{(\omega)}}(A)[\mathcal M^{(i)}])$
\item[] $=\delta_{\omega(3)}(\cl_{\mathcal M^{(\omega)}}(A)[\mathcal M^{(\omega)}])$
\item[] $=d_{\mathcal M^{(\omega)}}(\cl_{\mathcal M^{(\omega)}}(A))$
\item[] $=d_{\mathcal M^{(\omega)}}(A)$
\end{enumerate}

Thus $d^{(\omega)}(A)\leq d_{\mathcal M^{(\omega)}}(A)$.

Now we want to prove the inequality in the other direction.
Let $B$ be a finite set with $A\subseteq B\subseteq\mathcal M^{(i)}$. Let $B^{(i)}=B$. Suppose that we have constructed $B^{(j)}$ for some $j\geq i$, then we obtain $B^{(j+1)}$ from $B^{(j)}$ in the same manner as we obtain $\mathcal M^{(j+1)}$ from $\mathcal M^{(j)}$. More precisely, if $\bar c_j\nsubseteq B^{(j)}$ then $B^{(j+1)}=B^{(j)}$, if $\bar c_j\subseteq B^{(j)}$ then we remove $\bar c_j=(b_1,\cdots,b_n)$ and add the relations $\{(b_1,x_1,b_2),\cdots,(b_{n-1},x_{n-1},b_n)\}\cup\{x_1,\cdots,x_{n-1})\}$ where $\bar c_j'=(x_1,\cdots,x_{n-1})$. Note that $\{x_1,\cdots,x_{n-1}\}$ are really new points in $B^{(j+1)}\setminus B^{(j)}$ as $\{x_1,\cdots,x_{n-1}\}\cap\mathcal M^{(j)}=\emptyset$.

We have constructed $B^{(j)}$ with $i\leq j\in\omega$. Clearly the sequence stabilizes for some $k$ for which $B^{(k)}$ has only $3$-tuples (as in the proof of last claim). Moreover, $B^{(k)}$ is a substructure of $\mathcal M^{(\omega)}$. 

It is now easy to see that for each $j\geq i$ we have $$\delta_{\omega(3)}(B^{(j)}[\mathcal M^{(j)}])=\delta_{\omega(3)}(B^{(j+1)}[\mathcal M^{(j+1)}])$$ in particular we have, $$\delta_{\omega(3)}(B[\mathcal M^{(i)}])=\delta_{\omega(3)}(B^{(i)}[\mathcal M^{(i)}])=\delta_{\omega(3)}(B^{(k)}[\mathcal M^{(k)}])=\delta_{\omega(3)}(B^{(k)}[\mathcal M^{(\omega)}]).$$
The fact that we were able to find a finite substructure $B^{(k)}$ of $\mathcal M^{(\omega)}$ with $A\subseteq B^{(k)}\subseteq\mathcal M^{(\omega)}$ and $\delta_{\omega(3)}(B^{(k)}[\mathcal M^{(\omega)}])=\delta_{\omega(3)}(B[\mathcal M^{(i)}])$ proves that $d_{\mathcal M^{(\omega)}}(A)\leq d_{\mathcal M^{(i)}}(A)=d^{(\omega)}(A)$.

We have proved that $d^{(\omega)}(A)=d_{\mathcal M^{(\omega)}}(A)$ for all finite $A\subseteq\mathcal M^{(\omega)}$. Thus, $P^{(\omega)}=PG(\mathcal M^{(\omega)})$. 
\end{proof}

Now we observe that as $\mathcal M^{(\omega)}\in\bar{\mathcal C}_3$, then there exists a strong embedding $g:\mathcal M^{(\omega)}\stackrel{\leq}{\longrightarrow}\mathcal M_3$, (that is, $g$ is an embedding with $g(\mathcal M^{(\omega)})\leq\mathcal M_3$). For this we use the extension property (axiom $F2$) of $\mathcal M_3$ recursively ($\mathcal M^{(\omega)}$ is countable). But because $g$ is a strong embedding then $g$ is also an embedding of pregeometries, so we have that $PG(\mathcal M^{(\omega)})$ embeds in $PG(\mathcal M_3)$.

Finally, we recall all we have proved: $PG(\mathcal M_{\omega(3)})$ is a subpregeometry of $P^{(\omega)}$, $P^{(\omega)}=PG(\mathcal M^{(\omega)})$ and $PG(\mathcal M^{(\omega)})$ embeds in $PG(\mathcal M_3)$. Thus $PG(\mathcal M_{\omega(3)})$ embeds in $PG(\mathcal M_3)$.
\end{proof}

\begin{rem} \rm
Note that in contrast to the other results of this section, the previous lemma also holds if we work with unordered tuples.
\end{rem}

\begin{lem}
\label{L522}
Let $n\geq 3$ be a natural number. Then, $$PG(\mathcal M_3)\textit{ embeds in }PG(\mathcal M_n).$$
\end{lem}
\begin{proof}
Let $\mathcal M_3^h$ be the $L_n$-structure obtained from $\mathcal M_3$ by replacing each relation $(a_1,a_2,a_3)\in R_3^{\mathcal M_3}$ by the relation $(a_1,a_2,a_3,a_4,\cdots,a_n)$ where $a_i=a_3$ for all $i\geq 4$. Then, clearly, $\mathcal M_3^h\in\bar{\mathcal C}_n$ and $PG(\mathcal M_3^h)=PG(\mathcal M_3)$, because the predimension function is the same in both structures. But as $\mathcal M_3^h\in\bar{\mathcal C}_n$ we can use recursively the extension property to build a strong embedding $\mathcal M_3^h\stackrel{\leq}{\longrightarrow}\mathcal M_n$, that is, with $g(\mathcal M_3^h)\leq\mathcal M_n$. But then $g$ is also an embedding of pregeometries, thus $PG(\mathcal M_3)=PG(\mathcal M_3^h)$ embeds in $PG(\mathcal M_n)$. 
\end{proof}

We obtain the following theorem.
\begin{thm}
\label{T523}
Let $I$ be a countable set and $f:I\to\mathbbm N\setminus\{0\}\times\mathbbm N\setminus\{0\}$ be a function and $n\geq 3$ a natural number. Then there exists a chain of embeddings of pregeometries as in the following diagram: $$PG(\mathcal M_f)\to PG(\mathcal M_{\omega})\to PG(\mathcal M_3)\to PG(\mathcal M_n)\to PG(\mathcal M_{\omega}).$$
\end{thm}

\begin{proof}
By Corollary \ref{C517} we get $PG(\mathcal M_f)\to PG(\mathcal M_{\omega})$. By Theorem \ref{T516} we have $PG(\mathcal M_{\omega})\simeq PG(\mathcal M_{\omega(3)})$. By Lemma \ref{L520} we have that $PG(\mathcal M_{\omega(3)})$ embeds in $PG(\mathcal M_3)$. By Lemma \ref{L522}  we get $PG(\mathcal M_3)\to PG(\mathcal M_n)$. Finally, $\mathcal M_n\in\bar{\mathcal C}_n\subseteq\bar{\mathcal C}_{\omega}$ and $\mathcal M_n$ is countable, so we can use the extension property recursively to build an strong embedding $g:\mathcal M_n\to\mathcal M_{\omega}$. Then $g:PG(\mathcal M_n)\to PG(\mathcal M_{\omega})$ is an embedding of pregeometries. This conclude the proof.
\end{proof}

%We get the following corollary.
\begin{cor}\label{cor39}
Let $m,n\geq 3$ be natural numbers. Then $PG(\mathcal M_m)$ and $PG(\mathcal M_n)$ embed in each other. In particular, in the $LPI$ language of pregeometries we have that $PG(\mathcal M_m)$ and $PG(\mathcal M_n)$ have the same isomorphism types of finite subpregeometries.
\end{cor}
\begin{proof}
We use Theorem \ref{T523} to build a chain of embeddings $$PG(\mathcal M_m)\to PG(\mathcal M_{\omega})\to PG(\mathcal M_3)\to PG(\mathcal M_n).$$ 
%In the other direction we proceed similarly.
\end{proof}

\section{Pregeometries and different arities} 

In the previous section we have seen that for natural numbers $n,m\geq 3$  the pregeometries $PG(\mathcal M_n)$ and $PG(\mathcal M_m)$ embed in each other. The main result of this section is that if $m \neq n$, then $PG(\mathcal M_n)$ and $PG(\mathcal M_m)$ are not isomorphic.  In order to prove this we need two technical lemmas. The first of these lemmas is in fact almost trivial.
 
 \begin{lem}[First Changing Lemma]
\label{L525}
Keep the notation of the previous sections. In particular let $I$ be a countable set and $f:I\to\mathbbm N\setminus\{0\}\times\mathbbm N\setminus\{0\}$. Suppose $\mathcal M\in\bar{\mathcal C}_f$, $A\leq\mathcal M$ and $A'\in\bar{\mathcal C}_f$ where $A'$ has the same underlying set as $A$. Let $\mathcal M'$ be the structure obtained from $\mathcal M$ by replacing $A$ by $A'$, where by this we mean for each $i\in I$ we have $R_i^{\mathcal M'}=(R_i^{\mathcal M}\setminus R_i^A)\cup R_i^{A'}$. Then
$$A'\leq\mathcal M'\in\bar{\mathcal C}_f.$$
\end{lem}

\begin{proof}
Let $X$ be a finite subset of $\mathcal M'$. We want to prove that $\delta(X[\mathcal M'])\geq 0$.

We have that $A\leq\mathcal M\Rightarrow A\cap X[\mathcal M]\leq X[\mathcal M]$. Thus $\delta(X[\mathcal M])\geq\delta(A\cap X[\mathcal M])$. Notice that the points and relations added when going from $A\cap X[\mathcal M]$ to $X[\mathcal M]$ are the same points and relations added when going from $A'\cap X[\mathcal M']$ to $X[\mathcal M']$. Thus we also have $\delta(X[\mathcal M'])\geq\delta(A'\cap X[\mathcal M'])$. But as $A'\cap X[\mathcal M']$ is a substructure of $A'\in\bar{\mathcal C}_f$, we have $\delta(A'\cap X[\mathcal M'])\geq 0$. Thus $\delta(X[\mathcal M'])\geq 0$. This proves that $\mathcal M'\in\bar{\mathcal C}_f$. To see that $A'\leq\mathcal M'$ we use a similar argument.
\end{proof}

\begin{lem}[Second Changing Lemma]
\label{L526}
With the above notation, suppose $\mathcal M\in\bar{\mathcal C}_f$, $A\leq\mathcal M$ and $A'\in\bar{\mathcal C}_f$ where $A'$ has the same underlying set as $A$. Let $\mathcal M'$ be the structure obtained from $\mathcal M$ by replacing $A$ by $A'$ as we have done in Lemma \ref{L525}. Then $A'\leq\mathcal M'\in\bar{\mathcal C}_f$ and
if $PG(A)=PG(A')$, then $PG(\mathcal M)=PG(\mathcal M').$
\end{lem}

\begin{proof}
Assume that $PG(A)=PG(A')$, that is, $d_A=d_{A'}$. We want to prove that $PG(\mathcal M)=PG(\mathcal M')$, that is, $d_{\mathcal M}=d_{\mathcal M'}$. The idea of this proof is to reconstruct in small steps $\mathcal M$ starting with $A$ and $\mathcal M'$ starting with $A'$, and at each one of this steps the dimension function will be the same in both of these parallel constructions (reconstructions). More precisely, in the first step we add the remaining points, and in each subsequent step we add one of the remaining relations.

Let $R:=\bigcup_{i\in I}R_i^{\mathcal M}\setminus R_i^A=\bigcup_{i\in I}R_i^{\mathcal M'}\setminus R_i^{A'}$.

%(1{st} STEP:Adding points)

Let $\mathcal M_0$ be the structure obtained from $\mathcal M$ by removing all the relations in $R$, that is, the relations contained in $\mathcal M$ not entirely contained in $A$. Similarly, let $\mathcal M_0'$ be obtained from $\mathcal M'$ by removing all the relations in $R$, that is, the relations contained in $\mathcal M'$ not entirely contained in $A'$. 

Clearly we have $d_{\mathcal M_0}=d_{\mathcal M_0'}$. In fact let $B$ be a finite subset of both $\mathcal M$ and $\mathcal M'$ (which share the same underlying set) and let $X=B\setminus(A\cap B)$, then we have
\begin{enumerate}
\item [] $d_{\mathcal M_0}(B)=d_{\mathcal M_0}(A\cap B)+|X|$
\item [] $=d_{A}(A\cap B)+|X|$
\item [] $=d_{A'}(A\cap B)+|X|$
\item [] $=d_{\mathcal M_0'}(A\cap B)+|X|=d_{\mathcal M_0'}(B)$
\end{enumerate} 
thus $d_{\mathcal M_0}=d_{\mathcal M_0'}$.

%(OTHER STEPS: Adding relations one by one)
Now we add the relations that we have removed one by one in both $\mathcal M_0$ and $\mathcal M_0'$. Let $R=(r_i)_{i\in\kappa}$ be an enumeration of $R$, with $\kappa$ some cardinal.

For each $j\in\kappa$ let $\mathcal M_j$ be obtained from $\mathcal M_0$ by adding the relations $\{r_i:i\in j\}$ and $\mathcal M_j'$ be obtained from $\mathcal M_0'$ by adding the same set of relations $\{r_i:i\in j\}$. Clearly $\mathcal M_{\kappa}=\mathcal M$ and $\mathcal M_{\kappa}'=\mathcal M'$. We want to prove that $d_{\mathcal M_{\kappa}}=d_{\mathcal M_{\kappa}'}$. The proof is by transfinite induction.

Let $j=0$. In this case we already proved that $d_{\mathcal M_0}=d_{\mathcal M_0'}$.

Now we want to prove that $d_{\mathcal M_j}=d_{\mathcal M_j'}\Rightarrow d_{\mathcal M_{j+1}}=d_{\mathcal M_{j+1}'}$. This follows directly from the following claim:

\begin{claim}
Let $A_1,A_2\in\bar{\mathcal C}_f$ have the same underlying set and that $PG(A_1)=PG(A_2)$. Let $B_1$ and $B_2$ be the structures obtained by adding the same relation $r$ to $A_1$ and $A_2$. Assume further that $B_1,B_2\in\bar{\mathcal C}_f$. Then $PG(B_1)=PG(B_2)$.
\end{claim}
\begin{proof}
Let $X$ be a finite subset of $B_1=B_2$ (as sets). We know that $d_{A_1}(X)=d_{A_2}(X)$ and we want to prove that $d_{B_1}(X)=d_{B_2}(X)$.

Fix $i\in\{1,2\}$. We have either $$d_{B_i}(X)=d_{A_i}(X)\textit{ or }d_{B_i}(X)=d_{A_i}(X)-1$$ because we just add one relation.

Suppose we have $d_{B_i}(X)=d_{A_i}(X)-1$. Then there is a finite $Y$ with $X\subseteq Y\subseteq B_i$ such that $\delta(Y[A_i])=d_{A_i}(X)$ and $Y\supseteq r$. In particular $d_{A_i}(X)\leq d_{A_i}(Y)\leq\delta(Y[A_i])=d_{A_i}(X)$. Thus $d_{A_i}(Y)=d_{A_i}(X)$.

Conversely, if there exists a finite $Y$ with $X\subseteq Y\subseteq B_i$ such that $d_{A_i}(Y)=d_{A_i}(X)$ and $Y\supseteq r$, then there exists a finite $Z\supseteq Y\supseteq X$ such that $\delta(Z[A_i])=d_{A_i}(X)$. But as $Z\supseteq r$ we have $\delta(Z[B_i])=d_{A_i}(X)-1$, thus $d_{B_i}(X)\leq d_{B_i}(Y)\leq\delta(Z[B_i])=d_{A_i}(X)-1$ so we get $d_{B_i}(X)=d_{A_i}(X)-1$.

But now we have
\medskip
\begin{enumerate}
\item[] $d_{B_1}(X)=d_{A_1}(X)-1\Leftrightarrow$
\item[] $\Leftrightarrow\exists Y\supseteq X\cup r\text{ such that } d_{A_1}(X)=d_{A_1}(Y)$
\item[] $\Leftrightarrow\exists Y\supseteq X\cup r\text{ such that } d_{A_2}(X)=d_{A_2}(Y)$ (because $d_{A_1}=d_{A_2}$)
\item[] $\Leftrightarrow d_{B_2}(X)=d_{A_2}(X)-1$.
\end{enumerate}
But consequently, we also have $d_{B_1}(X)=d_{A_1}(X)$ if and only if $d_{B_2}(X)=d_{A_2}(X)$. Now as we have $d_{A_1}(X)=d_{A_2}(X)$, the previous argument proves that either way we have $d_{B_1}(X)=d_{B_2}(X)$. Thus $d_{B_1}=d_{B_2}$, that is, $PG(B_1)=PG(B_2)$.
\end{proof}
%($j$ limit ordinal)

Now we return to our recursion argument considering the case when $j$ is a limit ordinal. Let $j$ be a limit ordinal. Assume that $d_{\mathcal M_i}=d_{\mathcal M_i'}$ for all $i\in j$. We want to prove that $d_{\mathcal M_j}=d_{\mathcal M_j'}$.

Let $X$ be a finite subset of $\mathcal M_j=\mathcal M_j'$ (as sets). 

Let $i_0\in j$ be such that all the relations in $\cl_{\mathcal M_j}(X)$ are in $\mathcal M_{i_0}$, this is possible because $\cl_{\mathcal M_j}(X)$ contains only finitely many relations. Then,
$$d_{\mathcal M_{i_0}}(X)\leq d_{\mathcal M_{i_0}}(\cl_{\mathcal M_j}(X))\leq\delta(\cl_{\mathcal M_j}(X)[\mathcal M_{i_0}])=\delta(\cl_{\mathcal M_j}(X)[\mathcal M_j])=d_{\mathcal M_j}(X)$$ thus $d_{\mathcal M_{i_0}}(X)=d_{\mathcal M_j}(X)$.

We have proved that for all $i\in j$ big enough (with $X$ fixed) we have $d_{\mathcal M_i}(X)=d_{\mathcal M_j}(X)$. Similarly, for all $i\in j$ big enough we have $d_{\mathcal M_i'}(X)=d_{\mathcal M_j'}(X)$. Let $i\in j$ be big enough for both cases, then using the inductive hypothesis we get
$$d_{\mathcal M_j}(X)=d_{\mathcal M_i}(X)=d_{\mathcal M_i'}(X)=d_{\mathcal M_j'}(X).$$

This means that we proved that $d_{\mathcal M_j}=d_{\mathcal M_j'}$ for all $j\in\kappa+1$. Finally we have $PG(\mathcal M)=PG(\mathcal M_{\kappa})=PG(\mathcal M_{\kappa}')=PG(\mathcal M')$, as desired.
\end{proof}

We now have the tools necessary to prove that the isomorphism types of pregeometries arising from different arities are different.

\begin{thm}
\label{T527}
Let $n,m\in\mathbbm N\setminus\{0\}$ with $n<m$. Then, $$PG(\mathcal M_n)\ncong PG(\mathcal M_m).$$
\end{thm}

\begin{rem}\rm In fact the proof we shall give shows that if $A \in \Cb_m$ and $B \in \Cb_n$ and there is $X \leq A$ consisting of $m$ points related by a single relation, then $PG(A)$ and $PG(B)$ are non-isomorphic.
\end{rem}

\begin{proof}
Suppose that there is an isomorphism of pregeometries from $PG(\mathcal M_n)$ to $PG(\mathcal M_m)$. We first observe that without loss of generality, we can assume that $\mathcal M_n$ and $\mathcal M_m$ have the same underlying set and that the isomorphism of pregeometries is the identity map. We can write $d$ instead of $d_{\mathcal M_n}$ and $d_{\mathcal M_m}$. Let $\cl^d$ be the closure operator of the pregeometry $PG(\mathcal M_n)=PG(\mathcal M_m)$.

Let $X=\{a_1,\cdots,a_m\}$ be a subset of $\mathcal M_m$ such that $\widetilde X:=X[\mathcal M_m]\leq\mathcal M_m$ and $R_m^{\widetilde X}=\{(a_1,\cdots,a_m)\}$ with $a_1,\cdots,a_m$ distinct. This is possible by the extension property of $\mathcal M_m$. Let $\widehat X:=X[\mathcal M_n]$.

Now consider $Y=\cl_{\mathcal M_n}(X)$ and $\widehat Y:=Y[\mathcal M_n]$ and $\widetilde Y:=Y[\mathcal M_m]$. We have $$d(Y)=d_{\mathcal M_n}(\cl_{\mathcal M_n}(X))=d_{\mathcal M_n}(X)=d(X)=d_{\mathcal M_m}(X)=\delta_m(X[\mathcal M_m])=m-1$$ thus $$d(Y)=d(X)=m-1.$$

Let $(a_{11},\cdots,a_{1n}),\cdots,(a_{k1},\cdots,a_{kn})$ be a list of the relations in $R_n^{\widehat Y}$. We have $\widehat Y=\cl_{\mathcal M_n}(X)\leq\mathcal M_n$, thus $$d(Y)=\delta_n(\widehat Y)=|Y|-k.$$

Now, for each $1\leq i\leq k$ let $Z_i:=\cl^d(\{a_{i1},\cdots,a_{in}\})$ and let $\widehat Z_i:=Z_i[\mathcal M_n]$ and $\widetilde Z_i:=Z_i[\mathcal M_m]$. Note that $Z_i$ can be infinite and that $\cl^d(Z_i)=Z_i$. Also we have $\widetilde Z_i\subseteq \cl_{\mathcal M_m}(\widetilde Z_i)\subseteq \cl^d(Z_i)$ and $\widehat Z_i\subseteq \cl_{\mathcal M_n}(\widehat Z_i)\subseteq \cl^d(Z_i)$. Thus we have $\widetilde Z_i=\cl_{\mathcal M_m}(\widetilde Z_i)$ and $\widehat Z_i=\cl_{\mathcal M_n}(\widehat Z_i)$, that is, $\widetilde Z_i\leq\mathcal M_m$ and $\widehat Z_i\leq\mathcal M_n$. In particular, the pregeometries $PG(\widetilde Z_i)$ and $PG(\widehat Z_i)$ are those induced from $PG(\mathcal M_m)=PG(\mathcal M_n)$ so we get $PG(\widetilde Z_i)=PG(\widehat Z_i)$.

Now observe that $Z_i\nsupseteq X$. In fact, if $X\subseteq Z_i$ then we would have $X\subseteq \cl^d(\{a_{i1},\cdots,a_{in}\})$, in particular $d(X)\leq d(\{a_{i1},\cdots,a_{in}\})$. But we have
\begin{enumerate}
\item[] $d(\{a_{i1},\cdots,a_{in}\})\leq\delta_n(\{a_{i1},\cdots,a_{in}\})[\mathcal M_n])$
\item[] $=|\{a_{i1},\cdots,a_{in}\}|-|R_n^{\{a_{i1},\cdots,a_{in}\}[\mathcal M_n]}|$
\item[] $\leq n-1<m-1=d(X).$
\end{enumerate}
Thus $Z_i\nsupseteq X$.

Let $Z_i^*\in\bar{\mathcal C}_m$ be the structure obtained from $\widehat Z_i$ by replacing each relation $(x_1,\cdots,x_n)\in R_n^{\widehat Z_i}$ by the relation $(x_1,\cdots,x_n,\cdots,x_n)$ with $m$ coordinates. Then we have $$PG(Z_i^*)=PG(\widehat Z_i)=PG(\widetilde Z_i).$$

Now we want to construct a sequence of $L_m$-structures $$\mathcal M_m=\mathcal M_m^{(0)},\cdots,\mathcal M_m^{(k)}$$ all with the same underlying set and with the following properties:
\begin{enumerate}
\item $\mathcal M_m^{(i)}\in\bar{\mathcal C}_m$
\item $PG(\mathcal M_m^{(i)})=PG(\mathcal M_m)$
\item $R_m^{\mathcal M_m^{(i)}}\supseteq\{(a_1,\cdots,a_m)\}\cup\{(a_{j1},\cdots,a_{jn},\cdots,a_{jn}):1\leq j\leq i\}.$
\end{enumerate}

Let $\mathcal M_m^{(0)}=\mathcal M_m$.  Then for $i=0$ the properties are satisfied, note that $(a_1,\cdots,a_m)\in R_m^{\mathcal M_m}$. Now suppose we have constructed $\mathcal M_m^{(i)}$ satisfying the properties and $i<k$. Then we obtain $\mathcal M_m^{(i+1)}$ from $\mathcal M_m^{(i)}$ by replacing $Z_{i+1}[\mathcal M_m^{(i)}]$ by $Z_{i+1}^*$. We need to prove that $\mathcal M_m^{(i+1)}$ satisfies the properties.

Observe that $Z_{i+1}[\mathcal M_m^{(i)}]\leq\mathcal M_m^{(i)}$. In fact we have $PG(\mathcal M_m^{(i)})=PG(\mathcal M_m)$ so we have $Z_{i+1}\subseteq \cl_{\mathcal M_m^{(i)}}[Z_{i+1}]\subseteq \cl^d(Z_{i+1})$ and as $Z_{i+1}=\cl^d(Z_{i+1})$ we get that $Z_{i+1}=\cl_{\mathcal M_m^{(i)}}(Z_{i+1})$, that is, $Z_{i+1}[\mathcal M_m^{(i)}]\leq\mathcal M_m^{(i)}$.

Now we apply the First Changing Lemma: $Z_{i+1}^*\in\bar {\mathcal C}_m$ and $Z_{i+1}[\mathcal M_m^{(i)}]\leq\mathcal M_m^{(i)}$ so $Z_{i+1}^*\leq\mathcal M_m^{(i+1)}\in\bar{\mathcal C}_m$, proving property $1)$.

For property $2)$ we use the Second Changing Lemma: all we need to prove is that $PG(Z_{i+1}^*)=PG(Z_{i+1}[\mathcal M_m^{(i)}])$. But in fact, $$d_{Z_{i+1}[\mathcal M_m^{(i)}]}=(d_{\mathcal M_m^{(i)}})_{|Z_{i+1}}=(d_{\mathcal M_m})_{|Z_{i+1}}=d_{Z_{i+1}[\mathcal M_m]}=d_{\widetilde Z_{i+1}}$$ thus $$PG(Z_{i+1}[\mathcal M_m^{(i)}])=PG(\widetilde Z_{i+1})=PG(\widehat Z_{i+1})=PG(Z_{i+1}^*).$$ Now we apply the Second Changing Lemma and we get $PG(\mathcal M_m^{(i+1)})=PG(\mathcal M_m^{(i)})=PG(\mathcal M_m)$, proving property $2)$.

Finally we need to prove property $3)$. First we observe that when going from $\mathcal M_m^{(i)}$ to $\mathcal M_m^{(i+1)}$ we add the relation $(a_{j1},\cdots,a_{jn},\cdots,a_{jn})$ for $j=i+1$, as this relation is in $R_m^{Z_{i+1}^*}$.

Observe also that $(a_1,\cdots,a_m)$ is not removed because $\{a_1,\cdots,a_m\}=X\nsubseteq Z_{i+1}$.

Now we need to see that for $j\leq i$ the relation $(a_{j1},\cdots,a_{jn},\cdots,a_{jn})$ is not removed when going from $\mathcal M_m^{(i)}$ to $\mathcal M_m^{(i+1)}$. If $\{a_{j1},\cdots,a_{jn}\}\nsubseteq Z_{i+1}$ this is clear. If $\{a_{j1},\cdots,a_{jn}\}\subseteq Z_{i+1}$ then
\begin{enumerate}
\item[] $(a_{j1},\cdots,a_{jn})\in R_n^{\mathcal M_n}\Leftrightarrow (a_{j1},\cdots,a_{jn})\in R_n^{Z_{i+1}[\mathcal M_n]}=R_n^{\widehat Z_{i+1}}$
\item[] $\Leftrightarrow (a_{j1},\cdots,a_{jn},\cdots,a_{jn})\in R_m^{Z_{i+1}^*}$
\end{enumerate}
thus $(a_{j1},\cdots,a_{jn},\cdots,a_{jn})$ is not removed. This proves that $\mathcal M_m^{(i+1)}$ satisfies property $3)$.

Finally we get the desired contradiction. $\mathcal M_m^{(k)}$ satisfies properties $2)$ and $3)$, that is, $PG(\mathcal M_m^{(k)})=PG(\mathcal M_m)$ and $$R_m^{Y[\mathcal M_m^{(k)}]}\supseteq\{(a_1,\cdots,a_m)\}\cup\{(a_{j1},\cdots,a_{jn},\cdots,a_{jn}):1\leq j\leq k\}.$$ Thus we can make the following calculation (remember that $d(Y)=|Y|-k$): $$d(Y)\leq\delta_m(Y[\mathcal M_m^{(k)}])=|Y|-|R_m^{Y[\mathcal M_m^{(k)}]}|\leq |Y|-(k+1)=(|Y|-k)-1=d(Y)-1.$$ Thus we get the desired contradiction $d(Y)\leq d(Y)-1$. We may then conclude that $PG(\mathcal M_n)\ncong PG(\mathcal M_m)$.
\end{proof}

\section{The local isomorphism type}

We have seen in the previous section that $PG(\mathcal M_n)$ is not isomorphic to $PG(\mathcal M_m)$. Here we show that $PG(\mathcal M_n)$ and $PG(\mathcal M_m)$ are not even locally isomorphic. First we need some more changing lemmas.

\begin{lem} (Third Changing Lemma)
\label{L528}
Let $f$ be as in  Lemma \ref{L525} and let $\mathcal M_f$ be the generic model for $(\mathcal C_f,\leq)$. Let $Z\leq\mathcal M_f$ with $Z$ finite. Let $Z'\in\mathcal C_f$ be a structure with the same underlying set as $Z$ and $\mathcal M_f'$ be obtained from $\mathcal M_f$ by replacing $Z$ by $Z'$, as we have done in Lemma \ref{L525}. Then we have $Z'\leq\mathcal M_f'\in\bar{\mathcal C}_f$ and 
$$\mathcal M_f'\simeq\mathcal M_f.$$
\end{lem}
\def\M{\mathcal{M}}
\begin{proof}
The only thing we have to show is that $\mathcal M_f'$ satisfies the extension property. For any set $X\subseteq\mathcal M_f$ we write $X$ instead of $X[\mathcal M_f]$ and $X'$ instead of $X[\mathcal M_f']$. Let $A\subseteq\mathcal M_f$ with $A'\leq\mathcal M_f'$ and $A'\leq B'\in\mathcal C_f$ ($B'$ arbitrary). We want to prove that there exists an embedding $g':B'\to\mathcal M_f'$ fixing the elements of $A'$ and with $g'(B')\leq\mathcal M_f'$. 

Let $W' = \cl_{\M_f'}(Z'A')$ and let $W$ be the corresponding structure in $\M_f$. Note that because $W' \leq \M_f'$ and $Z' \subseteq W'$ we have $W \leq \M_f$. Let $C'$ be the free amalgam of $W'$ and $B'$ over $A'$. Then $W', B' \leq C'$ and so we also have $Z', A' \leq C'$. Let $C$ be the result of changing $Z'$ to $Z$ in $C'$. Then we have $Z \leq C$ and (as above) $W \leq C$. We have $A \subseteq C$, but $A$ is not necessarily self-sufficient in $C$.

By the extension property for $\M_f$ there is an embedding $g : C \to \M_f$ with $g(C) \leq \M_f$ and $g$ fixing each element of $W$. Now consider the effect of changing $Z$ to $Z'$ in this embedding. The result is an embedding $g' : C' \to \M_f'$ (same map, different structures) with $g'(C') \leq \M_f'$ and which is the identity on $W'$. As $B' \leq C'$ we have $g'(B') \leq \M_f'$ (by transitivity), and  $g'$ is the identity on $A'$, as required. 
\end{proof}

Now we need a Fourth Changing Lemma concerning pregeometries and localization. For our purpose it would be enough to have a finite version of this lemma, however we decide to present a proof that works for the infinite case. We still show the finite case separately inside of the proof, the reason for this is that for the infinite case the proof requires the predimension to be of a particular form (but covering all the ones that we introduced), while the finite case could be more easily generalized to other contexts.

First some notation.

\begin{nota} \rm
Let $\mathcal M$ be a structure and $PG(\mathcal M)$ a pregeometry on $\mathcal M$ with closure operation $c$. Let $Z\subseteq\mathcal M$. We write $PG_Z(\mathcal M)$ to denote the localization of the pregeometry $PG(\mathcal M)$ to the set $Z$, that is, the pregeometry with closure operation $c_Z(A) = c(Z \cup A)$ for $A \subseteq \mathcal{M}$. If $c$ is given by the dimension function $d$ and $Z$ is finite, then $c_Z$ is given by the dimension function $d(X/Z) = d(X\cup Z) - d(Z)$.
\end{nota}

\begin{lem} (Fourth Changing Lemma)
\label{L530}
Let $f$ be as in Lemma \ref{L525} and let $\mathcal M\in\bar{\mathcal C}_f$. Let $Z\leq\mathcal M$ and $Z'\in\bar{\mathcal C}_f$ be a structure with the same underlying set as $Z$. Let $\mathcal M'$ be obtained from $\mathcal M$ by replacing $Z$ by $Z'$, as we have done in Lemma \ref{L525}. Then $Z'\leq\mathcal M'\in\bar{\mathcal C}_f$ and
$$PG_Z(\mathcal M)=PG_{Z'}(\mathcal M').$$
\end{lem}

\begin{proof}
Let us set up some notation. The value of the predimension depends on the structure. However, instead of distinguishing the structure from the underlying set (as in most of the paper), here we use different notation for the predimension. Depending on whether we are working inside of $\mathcal M$ or $\mathcal M'$, given a finite subset $A\subseteq\mathcal M$ we write
\begin{enumerate}
\item[] $\delta (A)=|A|-\sum_{i\in I}\alpha_i|R_i^{A[\mathcal M]}|$
\item[] $\delta'(A)=|A|-\sum_{i\in I}\alpha_i|R_i^{A[\mathcal M']}|$
\end{enumerate}
with $f(i)=(n_i,\alpha_i)$, where $n_i$ is the arity of $R_i$. Also we write $d$ instead of $d_{\mathcal M}$ and $d'$ instead of $d_{\mathcal M'}$.

First we give a separate proof for the case when $Z$ is finite.

\begin{claim}
Assume $Z$ is finite. Then $PG_Z(\mathcal M)=PG_{Z'}(\mathcal M')$.
\end{claim}
\begin{proof}
First observe that for every finite $A\subseteq\mathcal M$ we have $$\delta(A)-\delta(A\cap Z)=\delta'(A)-\delta'(A\cap Z)$$ because all the changes are done inside $Z$. 

All we have to prove is that $d(A/Z)=d'(A/Z)$. We do the following calculation:

\begin{enumerate}
\item[] $d(A/Z)=d(AZ)-d(Z)$
\item[] $=\delta(\cl^{\mathcal M}(AZ))-\delta(Z)$ (because $Z\leq\mathcal M$)
\item[] $=\delta'(\cl^{\mathcal M}(AZ))-\delta'(Z)$ (because $\cl^{\mathcal M}(AZ)\cap Z=Z$)
\item[] $\geq d'(\cl^{\mathcal M}(AZ))-d'(Z)$ (because $Z'\leq\mathcal M'$)
\item[] $\geq d'(AZ)-d'(Z)$
\item[] $=d'(A/Z)$
\end{enumerate}
We get $d(A/Z)\geq d'(A/Z)$. Proceeding analogously we get $d'(A/Z)\geq d(A/Z)$. We then obtain $d(A/Z)=d'(A/Z)$, concluding the proof of the claim. 
\end{proof}

Before proving the infinite case in full generality we need to prove the following weaker result, using an extra assumption.

\begin{claim}
Let $Z$ be possibly infinite, but $Z'$ obtained from $Z$ by just adding relations. Then $PG_Z(\mathcal M)=PG_{Z'}(\mathcal M')$.
\end{claim}
\begin{proof}
Our extra hypothesis means that for every set $B\subseteq\mathcal M$ we have $\cl_{\mathcal M}(B)\subseteq \cl_{\mathcal M'}(B).$ This is because from $\mathcal M$ to $\mathcal M'$ we are just adding relations, so  if we could decrease the predimension value of $B$ by taking an superset in $\mathcal M$, then the same superset would work as well in $\mathcal M'$.

Let $A$ be a finite subset of $\mathcal M$. Again we have $$\delta(A)-\delta(A\cap Z)=\delta'(A)-\delta'(A\cap Z).$$ Also, all we need to prove is that $d(A/Z)=d'(A/Z)$.

Let $Z_0$ be a finite subset of $Z$ such that $d(A/Z)=d(A/Z_0)$ and $d'(A/Z)=d'(A/Z_0)$: this can be done by choosing $Z_0$ large enough. Let $Z_1=\cl_{\mathcal M'}(AZ_0)\cap Z$. Then $Z_1\leq\mathcal M'$ and it is easy to see that $\cl_{\mathcal M'}(AZ_1)\cap Z=\cl_{\mathcal M}(AZ_1)\cap Z=Z_1$. 
In fact we have that $$Z_1\subseteq \cl_{\mathcal M}(AZ_1)\cap Z\subseteq \cl_{\mathcal M'}(AZ_1)\cap Z=Z_1.$$

Notice that we still have $d(A/Z)=d(A/Z_1)$ and $d'(A/Z)=d'(A/Z_1)$ as $Z_0\subseteq Z_1\subseteq Z$.

Now we are going to prove that we have $d(A/Z)=d'(A/Z)$ by proving the two inequalities separately.

\begin{enumerate}
\item[] $d(A/Z)=d(A/Z_1)$
\item[] $=d(AZ_1)-d(Z_1)$
\item[] $=\delta(\cl_{\mathcal M}(AZ_1))-\delta(Z_1)$ (because $Z_1\leq\mathcal M$)
\item[] $=\delta'(\cl_{\mathcal M}(AZ_1))-\delta'(Z_1)$ (because $\cl_{\mathcal M}(AZ_1)\cap Z=Z_1$)
\item[] $\geq d'(AZ_1)-d'(Z_1)$ (because $Z_1\leq\mathcal M'$)
\item[] $=d'(A/Z_1)$
\item[] $=d'(A/Z)$
\end{enumerate}

The same argument will prove the opposite inequality. Thus $d(A/Z)=d'(A/Z)$ for every finite $A\subseteq\mathcal M$. This proves that $PG_Z(\mathcal M)=PG_{Z'}(\mathcal M')$, proving the claim.
\end{proof}

Now we need to remove the extra assumption made in the previous claim in order to prove the full lemma.

For this we need to consider an intermediate structure $Z''$ with the same underlying set as $Z$ and $Z'$, with relations given by $R_i^{Z''}=R_i^Z\cap R_i^{Z'}$. Notice that $Z''\in\bar {\mathcal C}_f$ because we are just removing relations. Let $\mathcal M''$ be obtained from $\mathcal M$ by replacing $Z$ by $Z''$.

From $Z''$ to $Z$ we are just adding relations, thus by the previous claim we have $PG_{Z''}(\mathcal M'')=PG_Z(\mathcal M)$. But from $Z''$ to $Z'$ are also just adding relations, so again by the previous claim we have $PG_{Z''}(\mathcal M'')=PG_{Z'}(\mathcal M')$. Finally we can conclude that $PG_Z(\mathcal M)=PG_{Z'}(\mathcal M')$, proving the full lemma.
\end{proof}
\medskip
\medskip
\begin{cor}
\label{C531}
Let $f$ be as in Lemma \ref{L525} and let $\mathcal M\in\bar{\mathcal C}_f$. Let $Z\leq\mathcal M$ and $Z'\in\bar{\mathcal C}_f$ be a structure with the same underlying set as $Z$. Let $\mathcal M'$ be obtained from $\mathcal M$ by replacing $Z$ by $Z'$. Then $Z'\leq\mathcal M'\in\bar{\mathcal C}_f$ and
$$\textit{If }\cl^{d_{Z'}}(\emptyset)=Z'\textit{ then }PG_Z(\mathcal M)=PG(\mathcal M').$$
\end{cor}

\begin{proof}
Let $A$ be a finite subset of $\mathcal M'$. We have that $\cl^{d_{Z'}}(\emptyset)=Z'$ and $Z'\leq\mathcal M'$. Thus $Z'=\cl^{d_{Z'}}(\emptyset)=\cl^{d_{\mathcal M'}}(\emptyset)\cap Z'\subseteq \cl^{d_{\mathcal M'}}(A)$. So we have $\cl^{d_{\mathcal M'}}(A)=\cl^{d_{\mathcal M'}}(AZ')$, this means that $PG_{Z'}(\mathcal M')=PG(\mathcal M')$. But by the Fourth Changing Lemma we have $PG_Z(\mathcal M)=PG_{Z'}(\mathcal M')$, thus $PG_Z(\mathcal M)=PG(\mathcal M')$.
\end{proof}

Now we can say something about the localization in $PG(\mathcal M_n)$. More precisely we show that localizing over a finite subset does not change the isomorphism type of the pregeometry.

\begin{thm}
\label{T532}
Let $Z$ be a finite subset of $\mathcal M_n$. Then $$PG_Z(\mathcal M_n)\simeq PG(\mathcal M_n).$$
\end{thm}
\begin{proof} If $Y, Z \subseteq \M_n$ have the same $d$-closure then clearly $PG_Z(\M_n) = PG_Y(\M_n)$. Thus we may assume that $Z \leq \M_n$.  
Let $Z'$ be a structure in $\mathcal C_n$ with the same underlying set as $Z$ and $\delta(Z')=0$. This can be done: take for example $R_n^{Z'}=\{(c,\cdots,c):c\in Z'\}$. Let $\mathcal M_n'$ be obtained from $\mathcal M_n$ by replacing $Z$ by $Z'$. By the Third Changing Lemma we have $\mathcal M_n\simeq\mathcal M_n'$, in particular $PG(\mathcal M_n)\simeq PG(\mathcal M_n')$. But as $d_{Z'}(Z')=\delta(Z')=0$, by the Corollary \ref{C531} of the Fourth Changing Lemma we have $PG_Z(\mathcal M_n)=PG(\mathcal M_n')$, thus $PG_Z(\mathcal M_n)\simeq PG(\mathcal M_n)$.

\end{proof}

We finish this section with a consequence of the last theorem.

\begin{thm}
Let $m\neq n$ be natural numbers $\neq 0$. Then $PG(\mathcal M_m)$ and $PG(\mathcal M_n)$ are not locally isomorphic.
\end{thm}
\begin{proof}
By the last theorem, if they are locally isomorphic they must be isomorphic, which we know that is not true by Theorem \ref{T527}. Thus they are not locally isomorphic.
\end{proof}

\section{A generic structure for pregeometries}

The original motivation for this section was an attempt to prove that $PG(\mathcal M_n)$ and $PG(\mathcal M_m)$ \textit{are} isomorphic for $m,n\geq 3$. The idea was to try to recognize both pregeometries as a generic model of the same class of pregeometries and use the uniqueness of the generic model up to isomorphism: this idea was motivated by the fact that the isomorphism types of finite subpregeometries are the same in both pregeometries. Of course this idea cannot work as we have seen in Theorem \ref{T527} that $PG(\mathcal M_n)$ and $PG(\mathcal M_m)$ are not isomorphic. However we can in fact see $PG(\mathcal M_n)$ as a generic model of a class of pregeometries, but the appropriate \textit{embeddings} of the class change with $n$. 

\begin{defn} \rm
Let $n$ be a natural number greater than zero and let $(\mathcal C_n,\leq_n)$ be the usual amalgamation class corresponding to arity $n$. Here we use the notation $\leq_n$ instead of $\leq$ to distinguish self-sufficiency corresponding to different arities. We now consider $(\mathcal C_n,\leq_n)$ as a category where the objects are the elements of $\mathcal C_n$ and where the morphisms between objects $A,B\in\mathcal C_n$ are the strong embeddings between the $A$ and $B$, more precisely the embeddings $f:A\to B$ such that $f(A)\leq_n B$. 

Now we apply a functor to $(\mathcal C_n,\leq_n)$ that forgets the structure and remembers only the associated pregeometries and we obtain in this way a class of pregeometries $(P_n,\unlhd_n)$. More precisely, given two finite pregeometries $A\subseteq B$ we say that $A\unlhd_n B$ if and only if there are structures $\widetilde A,\widetilde B\in\mathcal C_n$ such that $PG(\widetilde A)=A$ and $PG(\widetilde B)=B$ and $\widetilde A\leq_n\widetilde B$. The objects are given by $P_n=\{A:A\mbox{ is a finite pregeometry and }\emptyset\unlhd_n A\}$. Also, we can see the category $(P_n,\unlhd_n)$ as a class of relational structures in the language $LPI$ of pregeometries, the morphisms are (some of the) embeddings in this language.
\end{defn}

\begin{rem} \rm
$P_n$ is closed under isomorphism, however not closed under substructures. Also $\unlhd_n$ is invariant under isomorphism, that is, if $f:B\to B'$ is an isomorphism of pregeometries and $A\subseteq B$ then we have $A\unlhd_n B$ if and only if $f(A)\unlhd_n f(B)$. 
\end{rem}

\begin{prop}
$P_3\subsetneq P_4$.
\end{prop}
\begin{proof}
If $A\in P_3$ then there is a structure $\widetilde A\in\mathcal C_3$ with $PG(\widetilde A)=A$. Now we change the structure $\widetilde A$ to $\widehat A$ by replacing each relation $(a,b,c)$ by the relation $(a,b,c,c)$. Notice that $\widehat A\in\mathcal C_4$ and $PG(\widehat A)=PG(\widetilde A)=A$, thus $A\in P_4$. We have proved that $P_3\subseteq P_4$.

However this is a proper inclusion. Consider a structure $B\in\mathcal C_4$ consisting of $4$ distinct points $a,b,c,d$ and only one relation $(a,b,c,d)$. The pregeometry associated to $B$ can be described by saying that for $X\subseteq B$ and $|X|\leq 3$ we have $d_B(X)=|X|$ and $d_B(B)=3$. We have $PG(B)\in P_4$. However, this pregeometry is not in $P_3$ because there is no structure in $\mathcal C_3$ matching this pregeometry. In fact, suppose that there is a structure $B'\in\mathcal C_3$ with $PG(B')=PG(B)$. We have $d_{B'}(B')=\delta(B')=3$, thus there is exactly one relation in $B'$, say it is $(x,y,z)$ with $x,y,z\in B'$. We would have $d_{B'}(\{x,y,z\})\leq\delta(\{x,y,z\})=|\{x,y,z\}|-1<|\{x,y,z\}|$ which does not happen in $PG(B)$. We have proved that $P_3\subsetneq P_4$.
\end{proof}

\begin{prop}
In the class $(P_n,\unlhd_n)$ the relation $\unlhd_n$ is transitive.
\end{prop}
\begin{proof}
Suppose that we have $A\unlhd_n B$ and $B\unlhd_n C$. Suppose that $\widehat A,\widehat B,\widetilde B,\widetilde C\in\mathcal C_n$ such that $PG(\widehat A)=A$, $PG(\widehat B)=PG(\widetilde B)=B$, $PG(\widetilde C)=C$ and such that $\widehat A\leq_n\widehat B$ and $\widetilde B\leq_n\widetilde C$.

Now we construct a structure $\widehat C$ obtained from $\widetilde C$ by replacing $\widetilde B$ by $\widehat B$. By the First Changing Lemma we have $\widehat B\leq_n\widehat C\in\mathcal C_n$ and as $PG(\widetilde B)=PG(\widehat B)$ then by the Second Changing Lemma we have $PG(\widehat C)=PG(\widetilde C)=C$. Now we have $\widehat A\leq_n\widehat B$ and $\widehat B\leq_n\widehat C$. Thus by transitivity of $\leq_n$ we get $\widehat A\leq_n\widehat C$ with $PG(\widehat A)=A$ and $PG(\widehat C)=C$, thus $A\unlhd_n C$.
\end{proof}

\begin{prop}
The class $(P_n,\unlhd_n)$ is an amalgamation class. More precisely, if we have $A_0\unlhd_n A_1\in P_n$ and $A_0\unlhd_n A_2\in P_n$ then there exists $P\in P_n$ and embeddings of pregeometries $f_i:A_i\to P$, for $i\in\{1,2\}$, such that $f_i(A_i)\unlhd_n P$.
\end{prop}
\begin{proof}
Suppose that we have $A_0\unlhd_n A_1\in P_n$ and $A_0\unlhd_n A_2\in P_n$. Let $\widehat A_0\leq_n\widehat A_1\in\mathcal C_n$ and $\widetilde A_0\leq_n\widetilde A_2\in\mathcal C_n$ be the corresponding lifts to $(\mathcal C_n,\leq_n)$. Let $\widehat A_2$ be the structure obtained from $\widetilde A_2$ by replacing $\widetilde A_0$ by $\widehat A_0$ then by the First Changing Lemma we have $\widehat A_0\leq_n\widehat A_2\in\mathcal C_n$, moreover, as $PG(\widetilde A_0)=PG(\widehat A_0)$ then by the Second Changing Lemma we have $PG(\widehat A_2)=PG(\widetilde A_2)=A_2$. Now we use the fact that $(\mathcal C_n,\leq_n)$ is an amalgamation class and the fact that $\widehat A_0\leq_n\widehat A_1\in\mathcal C_n$ and $\widehat A_0\leq_n\widehat A_2\in\mathcal C_n$ to construct a structure $C\in\mathcal C_n$ and embeddings $f_1:\widehat A_1\to C$ and $f_2:\widehat A_2\to C$ fixing elements of $\widehat A_0$ such that $f_1(\widehat A_1)\leq_n C$ and $f_2(\widehat A_2)\leq_n C$. We have $PG(f_i(\widehat A_i))\unlhd_n PG(C)$, that is, $f_i(A_i)\unlhd_n PG(C)\in P_n$. In other words, $f_i:A_i\to PG(C)$ are strong embeddings of pregeometries fixing elements of $A_0$ and $PG(C)\in P_n$. This proves that $(P_n,\unlhd_n)$ is an amalgamation class.
\end{proof}

We would like to define a notion of generic model for the class $(P_n,\unlhd_n)$. However $P_n$ is not closed under substructures but this is not a real barrier, we just need to adapt the definition of generic model to this context.

\begin{defn} \rm
\label{D540}
Let $(\mathcal C,\leq)$ be a class of finite relational structures in a countable language $L$ with countably many isomorphism types. Let $\leq$ a binary relation such that $A\leq B$ implies that $A$ is a substructure of $B$. Assume that $\mathcal C$ is closed under isomorphism (but not necessarily under substructures) and that $\leq$ is invariant under isomorphism. We say that an $L$-structure $\mathcal M$ is a generic model for $(\mathcal C,\leq)$ if $\mathcal M$ is countable and satisfies $FC 1$ and $FC 2$ where:
\begin{enumerate}
\item[$FC1$] There is a chain $M_0\leq M_1\leq M_2\leq\cdots$ with $M_i\in\mathcal C$ and $\bigcup_{i\in\mathbbm N}M_i=\mathcal M$.
\item[$FC2$] (Extension property) If $A\leq M_i$ and $A\leq B\in\mathcal C$ then there are $j\in\mathbbm N$ and an embedding $f:B\to M_j$ such that $f_{|A}=Id_{|A}$ and $f(B)\leq M_j$.
\end{enumerate}
\end{defn}

\begin{defn} \rm
Let $\mathcal M$ be a generic model for a class $(\mathcal C,\leq)$ as in the above definition, with respect to a chain $M_0\leq M_1\leq M_2\leq\cdots$ Let $A$ be a finite substructure of $\mathcal M$ such that $A\in\mathcal C$. We say that $A\leq\mathcal M$ if $A\leq M_i$ for some $i\in\mathbbm N$. Notice that, at least apparently, the set of $\leq$-subsets of $\mathcal M$ depends not only on the structure of $\mathcal M$ but also on the choice of the chain (so for example, it is not \textit{a priori} necessarily preseved by automorphisms).
\end{defn}

Now we state the existence and uniqueness up to isomorphism of the generic model. The proof of this result is only a straightforward modification of the standard procedure.

\begin{prop}
Let $(\mathcal C,\leq)$ be as in definition \ref{D540}. Assume that $\emptyset\leq A$ for all $A\in\mathcal C$, that $\leq$ is transitive and that $(\mathcal C,\leq)$ is an amalgamation class. Then there is a generic model for $(\mathcal C,\leq)$ in the sense of definition \ref{D540} and it is unique up to isomorphism.
\end{prop}

Now we can prove the main result of this section.

\begin{thm}
\label{T543}
There is a generic model $\mathcal P_n$ (unique up to isomorphism) for the class $(P_n,\unlhd_n)$ and $$\mathcal P_n\simeq PG(\mathcal M_n)$$ where $\mathcal M_n$ is the generic model for the class $(\mathcal C_n,\leq_n)$.
\end{thm}
\begin{proof}
There is a unique generic model $\mathcal P_n$ of the class $(P_n,\unlhd_n)$ because we have seen that this class satisfies the conditions of last proposition. To prove that $\mathcal P_n\simeq PG(\mathcal M_n)$ we just need to prove that $PG(\mathcal M_n)$ is also a generic model for the class $(P_n,\unlhd_n)$.

Let $A_0\leq_n A_1\leq_n A_2\leq_n\cdots$ with $A_i\in\mathcal C_n$ and $\mathcal M_n=\bigcup_{i\in\mathbbm N}A_i$. In particular we have $PG(A_0)\unlhd_n PG(A_1)\unlhd_n PG(A_2)\unlhd_n\cdots$ with $PG(A_i)\in P_n$ and $PG(\mathcal M_n)=\bigcup_{i\in\mathbbm N}PG(A_i)$. We want to prove that $PG(\mathcal M_n)$ is a generic model for $(P_n,\unlhd_n)$ with respect to this chain. It remains to prove the extension property.

Suppose that $A\unlhd_n PG(A_i)$ and $A\unlhd_n B\in P_n$. We want to prove that there exist $j\in\mathbbm N$ and an embedding of pregeometries $f:B\to PG(\mathcal M_n)$ fixing the elements of $A$ and with $f(B)\unlhd_n PG(A_j)$. We have $A\unlhd_n PG(A_i)$, so there are structures $A',A_i'\in\mathcal C_n$ such that $PG(A')=A$ and $PG(A_i')=PG(A_i)$ and such that $A'\leq_n A_i'$. Let $\mathcal M_n'$ be the structure obtained from $\mathcal M_n$ by replacing $A_i$ by $A_i'$. Then by the First Changing Lemma we have $A_i'\leq_n\mathcal M_n'\in\bar{\mathcal C}_n$, by the Second Changing Lemma we have $PG(\mathcal M_n')=PG(\mathcal M_n)$ and by the Third Changing Lemma we have $\mathcal M_n'\simeq\mathcal M_n$.

We have $A\unlhd_n B\in P_n$, so there are structures $\widetilde A,\widetilde B\in\mathcal C_n$ such that $PG(\widetilde A)=A$ and $PG(\widetilde B)=B$ and $\widetilde A\leq_n\widetilde B\in\mathcal C_n$. Now we construct $B'$ obtained from $\widetilde B$ by replacing $\widetilde A$ by $A'$, by the the First Changing Lemma we get $A'\leq_n B'\in\mathcal C_n$, moreover, as $PG(A')=A=PG(\widetilde A)$ we can apply the Second Changing Lemma and we get $PG(B')=PG(\widetilde B)=B$. 

Now we have $A'\leq_n A_i'\leq_n\mathcal M_n'$ and $A'\leq_n B'\in\mathcal C_n$. But we have $\mathcal M_n'\simeq\mathcal M_n$ so $\mathcal M_n'$ satisfies the extension property. We can apply the extension property to construct an embedding $f:B'\to\mathcal M_n'$ fixing the elements of $A'$ and such that $f(B')\leq_n\mathcal M_n'$. Now we apply the forgetful functor to this embedding and obtain an embedding $f:B\to PG(\mathcal M_n')=PG(\mathcal M_n)$ of pregeometries. It remains to prove that $f(B)\unlhd_n PG(\mathcal M_n)$. 

For $j\geq i$ let $A_j'$ be the substructure of $\mathcal M_n'$ with the same underlying set as $A_j$. Notice that $A_i\leq_n A_{i+1}$ and $PG(A_i')=PG(A_i)$ imply by the first and second Changing Lemmas that $A_i'\leq_n A_{i+1}'$ and $PG(A_{i+1}')=PG(A_{i+1})$. Then in the next step $A_{i+1}\leq_n A_{i+2}$ and $PG(A_{i+1}')=PG(A_{i+1})$ imply that $A_{i+1}'\leq_n A_{i+2}'$ and $PG(A_{i+2}')=PG(A_{i+2})$... We repeat this procedure a countable number of times and we obtain a chain $A_i'\leq_n A_{i+1}'\leq_n A_{i+2}'\leq_n\cdots\leq_n A_j'\leq_n\cdots$ with $PG(A_j')=PG(A_j)$ and $A_j'\leq_n\mathcal M_n'$ because we have $\mathcal M_n'=\bigcup_{j\geq i}A_j'$. 

Finally we have that $f(B')\leq_n\mathcal M_n'\Rightarrow f(B')\leq_n A_j'\mbox{ (for some } j\geq i)\Rightarrow f(B)=PG(f(B'))\unlhd_n PG(A_j')=PG(A_j)$. Thus $f(B)\unlhd_n PG(\mathcal M_n)$ as desired. We proved the extension property for $PG(\mathcal M_n)$ so we can conclude that $PG(\mathcal M_n)\simeq\mathcal P_n$.
\end{proof}

We end this section mentioning an alternative and more natural way of defining the generic model of $(P_n,\unlhd_n)$. However we do not know if such a generic model exists.

\begin{defn}\rm
Let $$\bar P_n:=\{A:A\mbox{ is pregeometry and there exists }A'\in\bar{\mathcal C}_n\mbox{ with }PG(A')=A\}$$ and let $P_n$ be as usual, that is, the class of finite pregeometries of $\bar P_n$. 

Now we define an embedding relation $\sqsubseteq_n$ on $\bar P_n$ by saying: $A\sqsubseteq_n B$ if and only if there are $A',B'\in\bar{\mathcal C}_n$ such that $PG(A')=A,PG(B')=B$ and $A'\leq_n B'.$\end{defn}

 It is clear that this is preseved by automorphisms of $B$.
Notice also that for finite pregeometries $\sqsubseteq_n$ coincides with $\unlhd_n$. In other words we have $(P_n,\sqsubseteq_n)=(P_n,\unlhd_n)$.

We can now give the alternative possible definition of generic model of $(P_n,\unlhd_n)$.

\begin{defn}\rm
Let $\mathcal M$ be a countable pregeometry. We say that $\mathcal M$ is a $\sqsubseteq_n$-generic model of $(P_n,\unlhd_n)$ if:

\begin{itemize}
\item $\mathcal M\in\bar P_n$
\item ($\sqsubseteq_n$-extension property) $A\sqsubseteq_n\mathcal M$, $A\sqsubseteq_n B\in P_n$ imply that there exists an embedding of pregeometries $f:B\to\mathcal M$ over $A$ such that $f(B)\sqsubseteq_n\mathcal M$.
\end{itemize}
\end{defn}

In fact, if the $\sqsubseteq_n$-generic of $(P_n,\unlhd_n)$ exists, then it is isomorphic to the $\unlhd_n$-generic, that is, isomorphic to $PG(\mathcal M_n)$. The following problem arises naturally:

\begin{problem}
Does the $\sqsubseteq_n$-generic model exists? In other words, does $PG(\mathcal M_n)$ satisfies the $\sqsubseteq_n$-extension property?
\end{problem}

\end{document}